\documentclass[reqno, 11pt]{amsart}

\usepackage[normalem]{ulem}%pour barrer du texte
\usepackage{diagbox}
\usepackage{amsmath,amssymb,color}
\usepackage{amsfonts, amscd, epsfig,enumerate,bm}
\usepackage{graphicx}
\usepackage{graphics}
\usepackage{mathrsfs}
\usepackage{todonotes}
\usepackage{soul}
\usepackage{enumitem}

\usepackage[displaymath, mathlines,pagewise]{lineno}
\usepackage[usenames,dvipsnames]{pstricks}
\usepackage{pstricks-add}
\usepackage{epsfig}
\usepackage{pst-grad} % For gradients
\usepackage{pst-plot} % For axes
\usepackage[space]{grffile} % For spaces in paths
\usepackage{etoolbox} % For spaces in paths
\usepackage[margin=3cm]{geometry}
\makeatletter % For spaces in paths
\patchcmd\Gread@eps{\@inputcheck#1 }{\@inputcheck"#1"\relax}{}{}
\makeatother

\newtheorem{theorem}{Theorem}[section]
\newtheorem{assumption}{Assumption}[section]
\newtheorem{remark}{Remark}[section]
\newtheorem{definition}{Definition}[section]

\newtheorem{lemma}[theorem]{Lemma}
\newtheorem{proposition}[theorem]{Proposition}

\usepackage{tikz}
\usetikzlibrary{backgrounds}
\usetikzlibrary{patterns,fadings}
\usetikzlibrary{arrows,decorations.pathmorphing}
\usetikzlibrary{decorations}
\usetikzlibrary{calc}
\usetikzlibrary{shapes.misc}

\usepackage{setspace}
\definecolor{light-gray}{gray}{0.95}
\usepackage{float}
\usepackage[colorlinks=true,linkcolor=blue,citecolor=magenta]{hyperref}
\def\centerarc[#1](#2)(#3:#4:#5){\draw[#1] ($(#2)+({#5*cos(#3)},{#5*sin(#3)})$) arc (#3:#4:#5);}

\numberwithin{equation}{section}
\numberwithin{figure}{section}

\newcommand{\mc}[1]{{\mathcal #1}}

\newcommand{\<}{\big\langle}
\renewcommand{\>}{\big\rangle}

\renewcommand{\epsilon}{\varepsilon}

\newcommand{\R}{\mathbb R}
\newcommand{\Z}{\mathbb Z}
\newcommand{\N}{\mathbb N}
\renewcommand{\P}{\mathbb P}

\newcommand{\E}{\mathbb E}

\allowdisplaybreaks[4]

\title[Stationary fluctuations for exclusion processes]{Stationary fluctuations for the WASEP with long jumps and infinitely extended reservoirs}

\author{Wenxuan Chen}
\address{School of Statistics and Data Science, Ningbo University of Technology, Ningbo, China.}
\email{wxchen@nbut.edu.cn}

\author{Linjie Zhao}
\address{School of Mathematics and Statistics, Huazhong University of Science \& Technology, Wuhan, China.}
\email{linjie\_zhao@hust.edu.cn}

\thanks{Acknowledgement. Zhao thanks the financial support from  the National Natural Science Foundation of China with grant numbers 12401168, 12371142 and 11971038, and the Fundamental Research Funds for the Central Universities in China.}

\keywords{exclusion process; long jumps; stationary fluctuations; stochastic Burgers equation.}

\begin{document}

\begin{abstract}
We study a weakly asymmetric exclusion process with long jumps and with infinitely many extended reservoirs. We prove that the stationary fluctuations of the process are governed by the generalized Ornstein-Uhlenbeck process or the stochastic Burgers equation with Dirichlet boundary conditions depending on the strength of the asymmetry of the dynamics.
\end{abstract}

\maketitle

\section{Introduction}

The Kardar-Parisi-Zhang equation was introduced in \cite{kardar1986dynamic} to describe the universal fluctuations of growing interfaces. Due to the non-linearity of the equation, it is ill-posed and has attracted a lot of attention in the mathematical society.   Different methods have been  proposed to solve the equation:  the rough path theory developed by \cite{hairer2013solving}, derivation of  the Cole-hopf solution from the weakly asymmetric simple exclusion process (WASEP) by Bertini and Giacomin \cite{bertini1997stochastic}, and derivation of  the stationary energy solution from the WASEP by Gon{\c c}alves and Jara \cite{gonccalves2014nonlinear}.  Recently, Matetski, Quastel and Remenik \cite{Mateski2021KPZ} constructed a Markov process, the KPZ fixed point, via the limit of rescaled height function for TASEP. This process is believed to be the universal limit for models in KPZ universality class.
It was shown in \cite{Quastel2023longKPZ}  that finite range asymmetric exclusion processes converge to the KPZ fixed point. A natural extension is to consider the  Kardar-Parisi-Zhang equation  with boundary conditions. In \cite{corwin2018open}, Corwin and Shen derived the KPZ equation on the unit interval with Neumann boundary conditions by considering the height functions of the asymmetric exclusion process on an interval.  Gon{\c c}alves \emph{et al.} \cite{gonccalves2020derivation} derived the stochastic Burgers equation, which can be roughly viewed as the gradient of the Kardar-Parisi-Zhang equation, with Dirichlet boundary conditions  from the boundary driven WASEP. See \cite{corwin2012kpzun,corwin2022some,quastel2011kpz} and references therein for an excellent review on this topic.

The aim of this article is to derive the stochastic Burgers equation from other interacting particle systems. Precisely, we study the boundary driven WASEP with long jumps. Recently, particle systems with long jumps and infinitely extended reservoirs have been extensively investigated.  Hydrodynamic limits for such models are closely related to the fractional Laplacian with various boundary conditions, see \cite{bernardin2017fractional, bernardin2019slow,bernardin2021microscopic} and references therein.  For fluctuations, Bernardin \emph{et al.} \cite{bernardin2022equilibrium} studied a symmetric exclusion process with long jumps and infinitely extended reservoirs in the diffusive regime, where they derived generalized Ornstein-Uhlenbeck processes with various boundary conditions. 

The dynamics we considered is composed of two parts: the first one is the same as in \cite{bernardin2022equilibrium}, where particles jump on the segment $\Lambda_n = \{1,2,\ldots,n-1\}$ according to some symmetric transition rate  with long jumps subject to the exclusion rule, and the boundary of $\Lambda_n$ acts as infinitely extended particle reservoirs.  Precisely, the symmetric jump rate is defined as $s_\alpha (x) = |x|^{-1-\alpha}$ for some $\alpha > 0$. On top of the symmetric dynamics, we add an asymmetric dynamics with jump rate given by $p_{\gamma} (x) = |x|^{-1-\gamma} \mathbf{1}_{\{x > 0\}}$ for $\gamma > 0$.  We assume the particle densities of the reservoirs are $1/2$.  See Figure \ref{fig: dynamics} for dynamics of the model. We mainly focus on the case $\alpha > 2$. Let the symmetric dynamics accelerate by $n^2$, and the asymmetric dynamics by $n^\theta$ for some $\theta > 0$.  Under Assumption \ref{assump: parameters} below, we prove that if $\theta < 3/2$, then the limit of the density fluctuation fields is given by the Ornstein-Uhlenbeck process with Dirichlet boundary conditions, while if $\theta = 3/2$, then the limit is given by the stochastic Burgers equation with Dirichlet boundary conditions. 

The main novelties of the article are twofold. First, inspired by the work in \cite{bernardin2022equilibrium},  we proposed a different definition to the solution of the stochastic Burgers equation  with Dirichlet boundary conditions (see Definition \ref{def burgers 2}), which is equivalent to the one in \cite{gonccalves2020derivation} (see Definition \ref{def burgers}). Moreover, Definition \ref{def burgers 2} is quite natural from the model we considered. Second, the critical step of the proof is to prove the second-order Boltzmann-Gibbs principle, see Proposition \ref{pro: A4}, which was first introduced in  \cite{gonccalves2014nonlinear} and has been shown to hold for a large class of interacting particle systems, see \cite{sethuraman2016microscopic,gonccalves2018density,diehl2017kardar} for example.  Compared with previous works, 
 the boundary terms are more involved in this model due to the long jumps, and we need to take advantage of the test functions.

We finally remark that it is challenging for us to consider the case $\alpha < 2$, from which we expect to derive fractional stochastic Burgers equation with Dirichlet boundary conditions.  One of the problems for us is that we cannot find a proper space to define the solution to the limiting SPDE. 

The article is organized as follows.  We state the model and results in Section \ref{sec model}.  The proof is outlined in Section \ref{sec proof}. Section \ref{sec: pf second order boltzmann gibbs} is devoted to the proof of the second order Boltzmann-Gibbs principle, \emph{i.e.}, Proposition \ref{pro: A4}, which is the main step of the proof. Finally, we prove the tightness of the fluctuation field and the uniqueness to the solution of the limiting SPDE in Sections \ref{sec: tightness} and \ref{sec: pf definitions equivalent} respectively.

\section{Model and results}\label{sec model}
  
The state space of the exclusion process is $\mc{X}_n = \{0,1\}^{\Lambda_n}$, where $\Lambda_n = \{1,2,\ldots,n-1\}$. Fix parameters $\theta, \gamma, \alpha > 0$. The generator of the process is 
\[L_n = n^2 L_{s} + n^\theta L_a.\]
Above, for $f: \mc{X}_n \rightarrow \R$,
\begin{align*}
L_{s} f(\eta) &= \sum_{x,y \in \Lambda_n}  s_{\alpha} (x-y)[f(\eta^{x,y}) -f(\eta)] \\
&+ \frac{1}{2}\sum_{x \in \Lambda_n, y \notin \Lambda_n}  s_{\alpha} (x-y) [f(\eta^{x}) -f(\eta)],\\
L_{a} f(\eta) &= \sum_{x,y \in \Lambda_n}  p_{\gamma} (y-x)  \eta_x (1-\eta_y)[f(\eta^{x,y}) -f(\eta)] \\
&+ \frac{1}{2}\sum_{x \in \Lambda_n, y \leq 0}  p_{\gamma} (x-y) (1-\eta_x) [f(\eta^{x}) -f(\eta)]\\
&+ \frac{1}{2}\sum_{x \in \Lambda_n, y \geq n}  p_{\gamma} (y-x) \eta_x [f(\eta^{x}) -f(\eta)],
\end{align*}
where $\eta^{x,y}$ is the configuration obtained from $\eta$ after swapping the values of $\eta_x$ and $\eta_y$,  and $\eta^x$ is the one after flipping the value of $\eta_x$,
\[\eta^{x,y}_z = \begin{cases}
\eta_x, \quad &\text{if } z = y,\\
\eta_y, \quad &\text{if } z = x,\\
\eta_z, \quad &\text{otherwise},
\end{cases} \quad \eta^{x}_z = \begin{cases}
1- \eta_x, \quad &\text{if } z = x,\\
\eta_z, \quad &\text{otherwise}.
\end{cases} \]
The jump rates are given by
\[s_{\alpha} (x) = |x|^{-1-\alpha} \mathbf{1}_{\{x \neq 0\}}, \quad p_{\gamma} (x) = |x|^{-1-\gamma} \mathbf{1}_{\{x > 0\}}.\]
The factor $1/2$ in the definition of the generators corresponds to the particle density of the reservoirs.  See Figure \ref{fig: dynamics} for dynamics of the model.

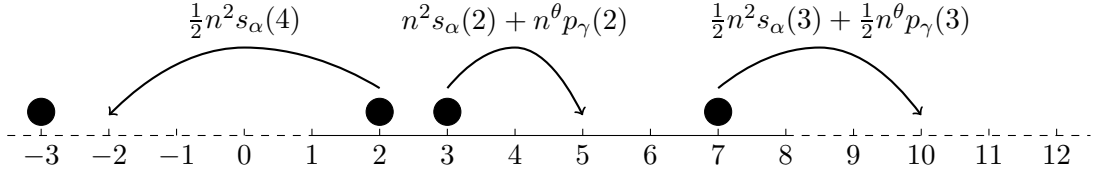
\begin{figure}[htbp]
	\begin{tikzpicture}[scale=0.9]
			\draw[-] (1,0) -- (8,0);
			\draw[dashed] (-3.5,0) -- (1,0);
			\draw[dashed] (8,0) -- (12.5,0);
			\foreach \x in {-3,-2,-1,0,1,2,3,4,5,6,7,8,9,10,11,12}
			\draw (\x,0) node[below]{$\x$}-- (\x,.1);
			
			\foreach \x in {-3,2,3,7}
			\draw[fill=black] (\x,.35)  circle (.2);
			
			\draw[->,thick] (2,.7) parabola bend (0,1.3) (-2,.3);
			\draw[->,thick] (3,.7) parabola bend (4,1.3) (5,.3);
			\draw[->,thick] (7,.7) parabola bend (8.5,1.3) (10,.3);
			
			\node at (0,1.7) {$\tfrac{1}{2}n^2 s_\alpha (4)$};
			\node at (4,1.7) {$n^2 s_\alpha (2) + n^\theta p_{\gamma} (2)$};
				\node at (8.8,1.7) {$\tfrac{1}{2}n^2 s_\alpha (3) + \tfrac{1}{2}n^\theta p_{\gamma} (3)$};
	\end{tikzpicture}
\caption{Dynamics of the model. Here, 	$n=9$. The sites on the dashed line are extended reservoirs.}
\label{fig: dynamics}
\end{figure}

\begin{remark}
The reservoir density $1/2$ is important since under the Bernoulli product measure with constant density $1/2$, the center of the mass travels at speed zero. Otherwise, we are not aware of how to choose the proper space of test functions in Subsection \ref{subsec: spde}. 
\end{remark}

We make the following assumptions on the above parameters throughout the article.

\begin{assumption}\label{assump: parameters} $\alpha > 2$, $\theta < \gamma \wedge 2$.
\end{assumption}

\begin{remark}
		For $0 < \alpha < 2$ or $\theta = \gamma$, the fractional Laplacian would appear in the limiting SPDE. In this case, we cannot find a proper space to define the solution to the fractional Ornstein-Uhlenbeck process or the fractional stochastic Burgers equation with boundary conditions. That is the reason why we  made the above assumptions on the parameters.
\end{remark}

For $\rho \in (0,1)$, let $\nu_\rho$ be the product measure on the state space $\mathcal{X}_n$ with constant particle density $\rho$,
\[\nu_\rho (\eta_x = 1) = \rho, \quad x \in \Lambda_n.\]
One could check directly that $\nu_{1/2}$ is reversible for the generator $L_s$, and is invariant for $L_a$.

Denote by $\eta(t) \equiv \eta^n (t)$ the Markov process with generator $L_n$ starting from the measure $\nu_{1/2}$. Let $\mathbb{P} \equiv \mathbb{P}^n$ be the probability measure on the path space $\mathcal{D} ([0,\infty); \mathcal{X}_n)$ induced by the process $\eta (t)$, and by $\E \equiv \E^n$ the corresponding expectation.

\subsection{The limiting SPDE}\label{subsec: spde} In this subsection, we rigorously define the solution to the Ornstein-Uhlenbeck process and the stochastic Burgers equation with Dirichlet boundary conditions. Before that, we first introduce the space of test functions.  Denote $\N := \{0,1,2,\ldots\}$. Define
\begin{align*}
\mathcal{S} &= \left\{ H\in \mathcal{C}^\infty([0,1]): H^{(i)}(0) =  H^{(i)}(1) = 0, \;\forall i\in\mathbb{N} \right\},\\
\mathcal{S}_{\rm Dir} &= \left\{ H\in \mathcal{C}^\infty([0,1]): H^{(2i)}(0) =  H^{(2i)}(1) = 0, \;\forall i\in\mathbb{N} \right\},\\
\mathcal{S}_{\rm Neu} &= \left\{ H\in \mathcal{C}^\infty([0,1]): H^{(2i+1)}(0) =  H^{(2i+1)}(1) = 0, \;\forall i\in\mathbb{N} \right\}.
\end{align*}
We equip the above spaces with the family of seminorms $\{\|H^{(i)}\|_{L^\infty( [0,1])}, \; i \in \N\}$ and denote respectively by  $\mathcal{S}^\prime$, $\mathcal{S}_{\rm Dir}^\prime$, $\mathcal{S}_{\rm Neu}^\prime$ their topological duals. For $H,G \in L^2 ([0,1])$, denote 
\[\<H,G\>:= \int_0^1 H(u) G(u) du.\]

\subsubsection{The Ornstein-Uhlenbeck process with Dirichlet boundary conditions}

We first define the solution of the following Ornstein-Uhlenbeck process with Dirichlet boundary conditions:
\begin{equation}\label{ou}
\partial_t \mathcal{Y}_t = A \Delta \mathcal{Y}_t  + \sqrt{D} \nabla \Dot{\mathcal{W}_t},
\end{equation}
where $A,D > 0$ are two parameters and $\Dot{\mathcal{W}_t}$ is a standard $\mathcal{S}_{\rm Neu}^\prime$-valued white noise.

\begin{definition}[The first definition of the solution to \eqref{ou}]\label{def ou}
We say $\big\{\mathcal{Y}_t, t \in [0,T]\big\} \in \mathcal{C} \big([0,T]; \mathcal{S}_{\rm Dir}^\prime\big)$ is a stationary solution of \eqref{ou} if it satisfies the following conditions:
\begin{enumerate}
    \item for any $t \in [0,T]$, for any $H,G \in \mathcal{S}_{\rm Dir}$,
    \begin{equation}\label{stationary def}
    \E [\mathcal{Y}_t (H)\mathcal{Y}_t (G)] = \frac{D}{2A} \<H,G\>.
    \end{equation}
    \item for any $H \in \mathcal{S}_{\rm Dir}$,
    \[\mathcal{M}_t (H) := \mathcal{Y}_t (H) - \mathcal{Y}_0 (H) - A \int_0^t \mathcal{Y}_s (\Delta H) ds\]
    is a continuous martingale with quadratic variation 
    \[\<\mathcal{M} (H)\>_t = t D \|\nabla H\|_{L^2 ([0,1])}^2.\]
\end{enumerate}
\end{definition}

\begin{remark}\label{ou unique}
We refer the readers to \cite[Theorem 2.13]{bernardin2022equilibrium} or \cite[Proposition 3.1]{gonccalves2020derivation} for the uniqueness of the solution to \eqref{ou}.
\end{remark}

In the following, we shall use an alternative definition of the solution to \eqref{ou}.  The equivalence of the two definitions are proved in \cite[Theorem 2.13]{bernardin2022equilibrium}. For $\varepsilon > 0$ and $u \in [0,1]$, define $\iota_{\varepsilon,u}: [0,1] \rightarrow \R$ as
\[\iota_{\varepsilon,u} (v) = \begin{cases}
\varepsilon^{-1} \mathbf{1} \{u \leq v < u+\varepsilon\} \quad &\text{if } \;0 \leq u < 1-\varepsilon,\\
\varepsilon^{-1} \mathbf{1} \{u -\varepsilon \leq v < u\} \quad &\text{if } \;1-\varepsilon \leq u \leq  1.
\end{cases}\]

\begin{definition}[The second definition of the solution to \eqref{ou}]\label{def ou 2}
We say $\big\{\mathcal{Y}_t, t \in [0,T]\big\} \in \mathcal{C} \big([0,T]; \mathcal{S}^\prime\big)$ is a stationary solution of \eqref{ou} if it satisfies the following conditions:
\begin{enumerate}
    \item the process $\mathcal{Y}_t$ is stationary in the sense of \eqref{stationary def}.
    \item for any $H \in \mathcal{S}$,
    \[\mathcal{M}_t (H) := \mathcal{Y}_t (H) - \mathcal{Y}_0 (H) - A \int_0^t \mathcal{Y}_s (\Delta H) ds\]
    is a continuous martingale with quadratic variation 
    \[\<\mathcal{M} (H)\>_t = t D \|\nabla H\|_{L^2 ([0,1])}^2.\]
    \item for any $u \in \{0,1\}$,
    \begin{equation}\label{ou bc}
    \lim_{\varepsilon \rightarrow 0} \E \Big[ \sup_{0 \leq t \leq T} \Big( \int_0^t \mathcal{Y}_s (\iota_{\varepsilon,u}) ds\Big)^2\Big] = 0.
    \end{equation}
\end{enumerate}
\end{definition}

\begin{remark}
Note that the space of test functions in Definitions \ref{def ou} and \ref{def ou 2} are different. In Definition \ref{def ou}, the information of the boundary conditions is contained in the space of test functions, while in Definition \ref{def ou 2}, it is contained in Eqn. \eqref{ou bc}. 
\end{remark}

\subsubsection{The stochastic Burgers equation with Dirichlet boundary conditions}

Next, we define the solution of the following stochastic Burgers equation with Dirichlet boundary conditions:
\begin{equation}\label{burgers}
\partial_t \mathcal{Y}_t = A \Delta \mathcal{Y}_t  + B \nabla \mathcal{Y}_t^2 + \sqrt{D} \nabla \Dot{\mathcal{W}_t},
\end{equation}
where $A,B,D > 0$ and $\Dot{\mathcal{W}_t}$ is a standard $\mathcal{S}_{\rm Neu}^\prime$-valued white noise. 

\begin{definition}[The first definition of the solution to \eqref{burgers}]\label{def burgers}
We say the process  $\big\{\mathcal{Y}_t, t \in [0,T]\big\}$ in $ \mathcal{C} \big([0,T]; \mathcal{S}_{\rm Dir}^\prime\big)$ is a stationary energy solution of \eqref{burgers} if it satisfies the following conditions:
\begin{enumerate}
\item the process $\mathcal{Y}_t$ is stationary in the sense of \eqref{stationary def}.
\item the process $\mathcal{Y}_t$ satisfies the following energy estimate:  there exists some finite constant $C$ such that for any $H \in \mathcal{S}_{\rm Dir}$, for any $\delta < \varepsilon \in (0,1)$ and any $0 \leq s < t \leq T$,
\[\E \Big[ \Big( \mathcal{A}_{s,t}^\varepsilon (H) - \mathcal{A}_{s,t}^\delta (H) \Big)^2\Big] \leq C (t-s) \varepsilon \|\nabla H\|_{L^2 ([0,1])}^2,\]
where
\begin{equation}
\mathcal{A}_{s,t}^\varepsilon (H) := - \int_s^t \int_0^1 \Big( \mathcal{Y}_r \big( \iota_{\varepsilon,u}\big)\Big)^2 \nabla H (u) du \,dr. 
\end{equation}
\item for any $H \in \mathcal{S}_{\rm Dir}$,
    \[\mathcal{M}_t (H) := \mathcal{Y}_t (H) - \mathcal{Y}_0 (H) - A \int_0^t \mathcal{Y}_s (\Delta H) ds - B \mathcal{A}_t (H)\]
    is a continuous martingale with quadratic variation 
    \[\<\mathcal{M} (H)\>_t = t D \|\nabla H\|_{L^2 ([0,1])}^2,\]
    where $\mathcal{A}_t (H)$ is the $L^2 (\P)$-limit of $\mathcal{A}_{0,t}^\varepsilon (H)$ as $\varepsilon \rightarrow 0$.
\item the reversed process $\{\mathcal{Y}_{T-t}, 0 \leq t \leq T\}$ satisfies condition (3) with $B$ replaced by $-B$.
\end{enumerate}
\end{definition}

\begin{remark}
We refer the readers to \cite[Theorem 3.3]{gonccalves2020derivation} for the uniqueness of the stationary energy solution to \eqref{burgers}.
\end{remark}

Similar to Definition \ref{def ou 2}, below we give an alternative definition of the stationary energy  solution to \eqref{burgers}.

\begin{definition}[The second definition of the solution to \eqref{burgers}]\label{def burgers 2}
We say the process $\big\{\mathcal{Y}_t, t \in [0,T]\big\}$ in $\mathcal{C} \big([0,T]; \mathcal{S}^\prime\big)$ is a stationary energy solution of \eqref{burgers} if it satisfies the following conditions:
\begin{enumerate}
\item the process $\mathcal{Y}_t$ is stationary in the sense of \eqref{stationary def}.
\item the process $\mathcal{Y}_t$ satisfies the following energy estimate:  there exists some finite constant $C$ such that for any $H \in \mathcal{S}$, for any $\delta < \varepsilon \in (0,1)$ and any $0 \leq s < t \leq T$,
\[\E \Big[ \Big( \mathcal{A}_{s,t}^\varepsilon (H) - \mathcal{A}_{s,t}^\delta (H) \Big)^2\Big] \leq C (t-s) \varepsilon \|\nabla H\|_{L^2 ([0,1])}^2,\]
where
\begin{equation}\label{def-A-delta}
\mathcal{A}_{s,t}^\varepsilon (H) := - \int_s^t \int_0^1 \Big( \mathcal{Y}_r \big( \iota_{\varepsilon,u}\big)\Big)^2 \nabla H (u) du \,dr. 
\end{equation}
\item for any $H \in \mathcal{S}$,
    \[\mathcal{M}_t (H) := \mathcal{Y}_t (H) - \mathcal{Y}_0 (H) - A \int_0^t \mathcal{Y}_s (\Delta H) ds - B \mathcal{A}_t (H)\]
    is a continuous martingale with quadratic variation 
    \[\<\mathcal{M} (H)\>_t = t D \|\nabla H\|_{L^2 ([0,1])}^2,\]
    where $\mathcal{A}_t (H)$ is the $L^2 (\P)$-limit of $\mathcal{A}_{0,t}^\varepsilon (H)$ as $\varepsilon \rightarrow 0$.
\item for any $u \in \{0,1\}$,
    \begin{equation}
    \lim_{\varepsilon \rightarrow 0} \E \Big[ \sup_{0 \leq t \leq T} \Big(  \int_0^t \mathcal{Y}_s (\iota_{\varepsilon,u}) ds\Big)^2\Big] = 0.
    \end{equation}
\item the reversed process $\{\mathcal{Y}_{T-t}, 0 \leq t \leq T\}$ satisfies condition (3) with $B$ replaced by $-B$.
\end{enumerate}
\end{definition}

The following result shows that Definitions \ref{def burgers} and \ref{def burgers 2} coincide, whose proof is presented in Section \ref{sec: pf definitions equivalent}.

\begin{proposition}\label{pro: definitions equivalent}
	Let $\big\{\mathcal{Y}_t, t \in [0,T]\big\}$ be given as in Definition \ref{def burgers 2}.  Then, it can be extended to $\mathcal{C} \big([0,T]; \mathcal{S}_{\rm Dir}^\prime\big)$  satisfying the conditions in Definition \ref{def burgers}. In particular, the solution defined in Definition \ref{def burgers 2} is unique.
\end{proposition}
 
\subsection{Density fluctuation fields}  We are interested in the limit of the density fluctuation field of the process, which acts on functions $H\in\mathcal{S}$  as
\[\mc{Y}^n_t (H) = \frac{1}{\sqrt{n}} \sum_{x \in \Lambda_n} \bar{\eta}_x (t) H(x/n),\]
where $\bar{\eta}_x = \eta_x - 1/2$. The following is the main result of this article.

\begin{theorem}\label{thm: fluctuation}
Fix $T > 0$. Under Assumption \ref{assump: parameters}, the sequence of processes $\{\mc{Y}^n_t, t \in [0,T]\}_{n \geq 1}$ converges in distribution in $\mathcal{C} ([0,T]; \mathcal{S}^\prime)$ as $n \rightarrow \infty$ towards
\begin{enumerate}
\item if $\theta < 3/2$, the solution to the Ornstein-Uhlenbeck process with Dirichlet boundary conditions \eqref{ou} with $A=2\sum_{y=1}^\infty y^{1-\alpha}, D=\sum_{y=1}^\infty y^{1-\alpha} $;
\item if $\theta = 3/2$, the solution to the stochastic Burgers equation with Dirichlet boundary conditions \eqref{burgers} with $A=2\sum_{y=1}^\infty y^{1-\alpha}, B=\sum_{\gamma=1}^\infty y^{-\gamma}, D=\sum_{y=1}^\infty y^{1-\alpha} $.
\end{enumerate}
\end{theorem}

\section{Proof of Theorem \ref{thm: fluctuation}}\label{sec proof}

\subsection{Associated martingales}\label{subsec: calculation} For any $H \in \mathcal{S}$, by Dynkin's martingale,
\begin{align}
\mc{M}^n_t (H) &= \mc{Y}^n_t (H) - \mc{Y}^n_0 (H) - \int_0^t L_n \mc{Y}^n_s (H) ds\label{martingale}\\
\mc{M}^n_t (H)^2 &- \int_0^t  \big\{L_n \mc{Y}^n_s (H)^2 - 2 \mc{Y}^n_s (H) L_n \mc{Y}^n_s (H)\big\} ds\label{quadratic variation}
\end{align}
are both martingales. 

We first calculate the integral term in \eqref{martingale}. Below, we write $\eta_x (s)$ simply as $\eta_x$ when there are no confusions. For the symmetric part,
\begin{align*}
n^2 L_{s} \mc{Y}^n_s (H) 
&= \frac{n^2}{\sqrt{n}} \sum_{x,y \in \Lambda_n} s_{\alpha} (x-y) \big[ H(\tfrac{x}{n}) - H(\tfrac{y}{n}) \big] [\eta_y - \eta_x]  \\
&+ \frac{n^2}{2\sqrt{n}} \sum_{x \in \Lambda_n, y \notin \Lambda_n} s_{\alpha} (x-y) H(\tfrac{x}{n}) (1-2\eta_x).
\end{align*}
Denote
\begin{equation}\label{def-r-alpha}
r_\alpha (x) := \sum_{y \notin \Lambda_n} s_\alpha (x-y), \quad x \in \Lambda_n.
\end{equation}
Note that $1-2\eta_x = - 2 \bar{\eta}_x$. Then, $n^2 L_{s} \mc{Y}^n_s (H) = A^{n,1}_s (H)+ A^{n,2}_s (H)$, where
\begin{align}
A^{n,1}_s(H) &= \frac{n^2}{\sqrt{n}} \sum_{x,y \in \Lambda_n} s_{\alpha} (x-y) \big[ H(\tfrac{x}{n}) - H(\tfrac{y}{n}) \big] [\Bar{\eta}_y - \Bar{\eta}_x],\label{Def-A-1}\\
A^{n,2}_s(H) &= - \frac{n^2}{\sqrt{n}} \sum_{x \in \Lambda_n} r_\alpha (x) H(\tfrac{x}{n}) \Bar{\eta}_x. \label{Def-A-2}
\end{align}
For the asymmetric part,
\begin{align}
n^\theta L_{a} \mc{Y}^n_s (H) 
&= \frac{n^\theta}{\sqrt{n}} \sum_{x,y \in \Lambda_n} p_{\gamma} (y-x) \eta_x (1-\eta_y)  \big[ H(\tfrac{y}{n}) - H(\tfrac{x}{n}) \big]   \notag \\
&+ \frac{n^\theta}{2\sqrt{n}} \sum_{x \in \Lambda_n, y \leq 0} p_{\gamma} (x-y) (1-\eta_x) H(\tfrac{x}{n}) \notag\\
&- \frac{n^\theta}{2\sqrt{n}} \sum_{x \in \Lambda_n, y \geq n} p_{\gamma} (y-x) \eta_x H(\tfrac{x}{n}).\label{asym part}
\end{align}
Define $s_\gamma$ and $a_\gamma$ as the symmetric and anti-symmetric parts of $p_\gamma$ respectively,
\[s_\gamma (x) = \frac{p_\gamma (x) + p_\gamma (-x)}{2}, \quad a_\gamma (x) = \frac{p_\gamma (x) - p_\gamma (-x)}{2}.\]
Then, the first term on the righthand side of \eqref{asym part} equals
\begin{align*}
A^{n,3}_s(H) + \frac{n^\theta}{2\sqrt{n}} \sum_{x,y \in \Lambda_n} a_{\gamma} (y-x) [\eta_x - \eta_y]^2 \big[ H(\tfrac{y}{n}) - H(\tfrac{x}{n}) \big],
\end{align*}
where
\[A^{n,3}_s(H) = \frac{n^\theta}{2\sqrt{n}} \sum_{x,y \in \Lambda_n} s_{\gamma} (y-x) [\Bar{\eta}_x - \Bar{\eta}_y] \big[ H(\tfrac{y}{n}) - H(\tfrac{x}{n}) \big].\]
Since
\[[\eta_x - \eta_y]^2 = (\Bar{\eta}_x)^2 + (\Bar{\eta}_y)^2 - 2\Bar{\eta}_x\Bar{\eta}_y = \frac{1}{2} - 2\Bar{\eta}_x\Bar{\eta}_y, \]
the first term on the righthand side of \eqref{asym part} equals
\begin{align}
   &A^{n,3}_s(H) + A^{n,4}_s(H) + \frac{n^\theta}{4\sqrt{n}} \sum_{x,y \in \Lambda_n} a_{\gamma} (y-x) \big[ H(\tfrac{y}{n}) - H(\tfrac{x}{n}) \big]\notag\\
   &=  A^{n,3}_s(H) + A^{n,4}_s(H) + \frac{n^\theta}{4\sqrt{n}} \sum_{x,y \in \Lambda_n} [a_{\gamma} (x-y) - a_{\gamma} (y-x)]  H(\tfrac{x}{n}),\label{asym part 1}
\end{align}
where
\[A^{n,4}_s(H) = -\frac{n^\theta}{\sqrt{n}} \sum_{x,y \in \Lambda_n} a_{\gamma} ({y-x}) \big[ H(\tfrac{y}{n}) - H(\tfrac{x}{n}) \big] \Bar{\eta}_x\Bar{\eta}_y. \]
The sum of the second and the third terms on the righthand side of \eqref{asym part} equals
\begin{equation}\label{asym part 2}
    A^{n,5}_s(H) + \frac{n^\theta}{4 \sqrt{n}} \sum_{x \in \Lambda_n, y \leq 0} p_{\gamma} (x-y) H(\tfrac{x}{n}) - \frac{n^\theta}{4 \sqrt{n}} \sum_{x \in \Lambda_n, y \geq n} p_{\gamma} (y-x) H(\tfrac{x}{n}),
\end{equation}
where
\[A^{n,5}_s(H) = - \frac{n^\theta}{2 \sqrt{n}} \sum_{x \in \Lambda_n} r_\gamma (x) H(\tfrac{x}{n}) \bar{\eta}_x, \quad r_\gamma (x) = \sum_{y \leq 0} p_\gamma (x-y) + \sum_{y \geq n} p_\gamma (y-x).\]
Since
\begin{align*}
    \sum_{y \in \Lambda_n} [a_\gamma (x-y) - a_{\gamma} (y-x)] &= - \sum_{y \notin \Lambda_n} [a_\gamma (x-y) - a_{\gamma} (y-x)] \\
    &= - \sum_{y \notin \Lambda_n} [p_\gamma (x-y) - p_{\gamma} (y-x)]\\
    &= - \sum_{y \leq 0} p_\gamma (x-y) + \sum_{y \geq n} p_\gamma (y-x),
\end{align*}
 the constant terms in \eqref{asym part 1}\ and \eqref{asym part 2} cancel out each other. Thus,
\begin{equation}\label{sep-Ln}
L_n \mc{Y}^n_s (H) = \sum_{j=1}^5 A^{n,j}_s(H).
\end{equation}

For the integral term in \eqref{quadratic variation}, one can easily show that
\begin{align*}
    &L_n \mc{Y}^n_s (H)^2 - 2 \mc{Y}^n_s (H) L_n \mc{Y}^n_s (H) \\
    &=  \sum_{x,y \in \Lambda_n} [n s_\alpha (x-y) + n^{\theta-1} p_\gamma (y-x) \eta_x (1-\eta_y)] (\eta_x - \eta_y)^2 \big[ H(\tfrac{y}{n}) - H(\tfrac{x}{n}) \big]^2\\
    &+  \sum_{x \in \Lambda_n} \Big[  \frac{n}{2} r_\alpha (x) + \frac{n^{\theta-1}}{2} \sum_{y \leq 0} p_\gamma (x-y)(1-\eta_x) + \frac{n^{\theta-1}}{2} \sum_{y \geq n} p_\gamma (y-x)\eta_x \Big]  H(\tfrac{x}{n})^2.    
\end{align*}

Next, we deal with the above terms respectively.

The following result shows that we can replace the term $A^{n,1}_s$ by a constant multiple of  $\mathcal{Y}^n_s (\Delta H)$ as $n \rightarrow \infty$.

\begin{lemma}\label{lem: A^n_1} For any $H \in \mathcal{S}$,
	\[\lim_{ n \rightarrow \infty} \E \Big[  \sup_{0 \leq t \leq T} \Big(  \int_0^t \big\{ A^{n,1}_s (H) - C_\alpha \mathcal{Y}^n_s (\Delta H) \big\} ds \Big)^2\Big] = 0,\]
	where $C_\alpha :=2 \sum_{y =1}^\infty y^{1-\alpha}$.
\end{lemma}

\begin{proof}
Recall \eqref{Def-A-1}, since $s_\alpha$ is symmetric, 
\[A^{n,1}_s (H) = \frac{2n^2}{\sqrt{n}} \sum\limits_{x,y\in\Lambda_n} s_\alpha(x-y) \big[ H(\tfrac{y}{n})-H(\tfrac{x}{n})) \big]\Bar{\eta}_x. \]
For any $H \in \mathcal{S}$, let $\tilde{H} \in C_c^\infty (\R)$ be the extension of $H$ from $[0,1]$ to $\R$ such that $\tilde{H} (u) = 0$ for $u \notin [0,1]$. Then, 
\begin{equation*}
	\begin{aligned}
		A^{n,1}_s (H) & =\frac{2}{\sqrt{n}} \sum\limits_{x\in\Lambda_n} \Bar{\eta}_x \cdot n^2 \sum\limits_{ y\in\Z } |x-y|^{-1-\alpha} \big[ 
		\tilde{H}(\tfrac{y}{n}) - \tilde{H}(\tfrac{x}{n}) \big]\\
		& - \frac{2}{\sqrt{n}} \sum\limits_{x\in\Lambda_n} \Bar{\eta}_x \cdot n^2 \sum\limits_{ y \notin \Lambda_n } |x-y|^{-1-\alpha} \big[ 
		\tilde{H}(\tfrac{y}{n}) - \tilde{H}(\tfrac{x}{n}) \big].
	\end{aligned}
\end{equation*}
By Cauchy-Schwarz inequality, the lemma follows directly from the following estimates:
\begin{align}
	\lim\limits_{n\to\infty} E_{\nu_{1/2}} \Big[ \Big( \frac{2}{\sqrt{n}}\sum\limits_{x\in\Lambda_n} \bar{\eta}_x \cdot
n^2\sum\limits_{y\in\mathbb{Z}}|y-x|^{-1-\alpha}
[\tilde{H}(\tfrac{y}{n}) - \tilde{H}(\tfrac{x}{n})] 
- C_\alpha \mathcal{Y}^n_s (\Delta H) \Big)^2\Big] = 0, \label{A-1-est}\\
	\lim\limits_{n\to\infty} E_{\nu_{1/2}} \Big[ \Big( \frac{2}{\sqrt{n}} \sum\limits_{x\in\Lambda_n} \Bar{\eta}_x \cdot n^2 \sum\limits_{ y \notin \Lambda_n } |x-y|^{-1-\alpha} \big[ 
\tilde{H}(\tfrac{y}{n}) - \tilde{H}(\tfrac{x}{n}) \big] \Big)^2\Big] = 0.\label{A-1-est2}
\end{align}

We first prove \eqref{A-1-est}.  We first write
\begin{multline*}
n^2\sum\limits_{y\in\mathbb{Z}} |y-x|^{-1-\alpha}
\big[\tilde{H}(\tfrac{y}{n}) - \tilde{H}(\tfrac{x}{n})\big]\\
=  n^2\sum\limits_{|y-x| \leq n} |y-x|^{-1-\alpha}
\big[\tilde{H}(\tfrac{y}{n}) - \tilde{H}(\tfrac{x}{n})\big]
+ n^2\sum\limits_{|y-x| > n} |y-x|^{-1-\alpha}
\big[\tilde{H}(\tfrac{y}{n}) - \tilde{H}(\tfrac{x}{n})\big].
\end{multline*}
The second term on the right hand side has order $n^{2-\alpha}$, which converges to zero as $n \rightarrow \infty$ since $\alpha > 2$. For the first term on the right hand side, by Taylor's expansion, it equals
\begin{align*}
 \frac12\sum\limits_{|y| \leq n}|y|^{1-\alpha}\cdot 
\tilde{H}^{(2)}(\tfrac{x}{n}) + \mathcal{R}_n^\alpha,
\end{align*}
where 
\begin{align*}
|\mathcal{R}_n^\alpha| \leq C(H) n^{-2} \sum_{|y| \leq n} |y|^{3-\alpha}.
\end{align*}
Note that the  error term above has order $O(n^{-2})$ if $\alpha>4$, $O(n^{2-\alpha}\log n)$ if $\alpha=4$, $O(n^{2-\alpha})$ if $\alpha<4$, which converges to 0 as $n\to\infty$.  Thus,
\begin{equation}\label{eqn 7}
\lim_{ n \rightarrow \infty} \sup_{x \in \Lambda_n} \Big| n^2\sum\limits_{y\in\mathbb{Z}} |y-x|^{-1-\alpha}
\big[\tilde{H}(\tfrac{y}{n}) - \tilde{H}(\tfrac{x}{n})\big] - \frac{C_\alpha}{2} \tilde{H}^{(2)}(\tfrac{x}{n})  \Big| = 0.
\end{equation}
This proves \eqref{A-1-est} by Cauchy-Schwarz inequality.

It remains to prove \eqref{A-1-est2}.  Since $\tilde{H}(y)=0$ if $y\notin[0,1]$,  the term in the square can be rewritten as 
\begin{equation}\label{eqn 1}
	\frac{2n^2}{\sqrt{n}}\sum\limits_{x\in\Lambda_n}\bar{\eta}_x
	\tilde{H}(\tfrac{x}{n}) \sum\limits_{y\notin\Lambda_n}|x-y|^{-1-\alpha}.
\end{equation} 
Since
\[ \sum\limits_{y\notin\Lambda_n}|x-y|^{-1-\alpha} 
\leq C(\alpha)(x^{-\alpha} + (n-x)^{-\alpha}),\]
we have 
\begin{align*}
	&E_{\nu_{1/2}}\Big[\Big( \frac{2n^2}{\sqrt{n}}\sum\limits_{x\in\Lambda_n}\bar{\eta}_x
	\tilde{H}(\tfrac{x}{n}) \sum\limits_{y\notin\Lambda_n}|x-y|^{-1-\alpha} \Big)^2\Big]\\
	\leq & C(\alpha)n^3\sum\limits_{x\in\Lambda_n}\tilde{H}(\tfrac{x}{n})^2
	(x^{-2\alpha} + (n-x)^{-2\alpha})\\
	\leq & C(\alpha)n^{4-2\alpha}\int_{\frac{1}{n}}^{1-\frac{1}{n}}
	\tilde{H}(u)^2(u^{-2\alpha} + (1-u)^{-2\alpha})\,du.
\end{align*}
Since $H\in\mathcal{S}$, $\tilde{H} (u)^2 \leq C(H) [u^{2\alpha} \mathbf{1}_{\{u\in[0,\frac12]\}} + (1-u)^{2\alpha} \mathbf{1}_{\{u\in(\frac12,1]\}}]$. In particular,  the last integral is finite.  This concludes the proof since $\alpha>2$. 
\end{proof}

The next result shows that $A^{n,i}_s$ vanishes as $n \rightarrow \infty$ for $i=2,3,5$.

\begin{lemma}\label{lem: A^n_2} For $i=2,3,5$, 
\[\lim_{ n \rightarrow \infty} \E  \Big[  \sup_{0 \leq t \leq T} \Big(  \int_0^t A^{n,i}_s (H)  ds \Big)^2\Big] = 0.\]
\end{lemma}

\begin{proof}
Observe that the term $A^{n,2}_s$ is the same as \eqref{eqn 1}.  The term $A^{n,3}_s$ can be dealt with similarly to the term $A^{n,1}_s$. Indeed,  we first write
\begin{align*}
	A^{n,3} (H) &= \frac{n^\theta}{\sqrt{n}} \sum_{x,y \in \Z} s_{\gamma} (y-x) \Bar{\eta}_x \big[ \tilde{H}(\tfrac{y}{n}) - \tilde{H}(\tfrac{x}{n}) \big]  \\
	&- \frac{n^\theta}{\sqrt{n}} \Big(\sum_{x \in \Lambda_n, y \notin \Lambda_n} + \sum_{y \in \Lambda_n, x \notin \Lambda_n} \Big) s_{\gamma} (y-x) \Bar{\eta}_x \big[ \tilde{H}(\tfrac{y}{n}) - \tilde{H}(\tfrac{x}{n}) \big].
\end{align*}
For $0 < \gamma < 2$, we have
\begin{align*}
&\frac{1}{n} \sum_{x \in \Z} \Big( n^\gamma \sum_{y \in \Z} s_{\gamma} (y-x)\big[ \tilde{H}(\tfrac{y}{n}) - \tilde{H}(\tfrac{x}{n}) \big] \Big)^2 \\
&\leq C \int_{\R} du \Big(\int_{\R} \frac{\tilde{H}(u) - \tilde{H}(v)}{|u-v|^{1+\gamma}} dv\Big)^2 
\leq C \int_{\R} du \int_{\R} dv \frac{[\tilde{H}(u) - \tilde{H}(v)]^2}{|u-v|^{1+\gamma}}\\
&\leq C \iint_{|u-v| > 1} \frac{2 [\tilde{H}(u)^2 + \tilde{H}(v)^2]}{|u-v|^{1+\gamma}} du dv + C \iint_{|u-v| \leq  1}  \|\tilde{H}^\prime\|_{1,\infty} (v)^2 |u-v|^{1-\gamma} du dv\\
& \leq C(H),
\end{align*}
where $\|H\|_{1,\infty} (v) = \sup_{|u-v| \leq 1} |H(u)|$. For $\gamma \geq 2$, similar to the proof of \eqref{eqn 7},  we have
\[\frac{1}{n} \sum_{x \in \Z} \Big( n^2 \sum_{y \in \Z} s_{\gamma} (y-x)\big[ \tilde{H}(\tfrac{y}{n}) - \tilde{H}(\tfrac{x}{n}) \big] \Big)^2  \leq \begin{cases}
	C(H) \quad &\text{if}\quad \gamma > 2,\\
	C(H) (\log n)^2 \quad &\text{if}\quad \gamma = 2.
\end{cases}\]
Thus, we bound
\[E_{\nu_{1/2}} \Big[  \Big( \frac{n^\theta}{\sqrt{n}} \sum_{x,y \in \Z} s_{\gamma} (y-x) \Bar{\eta}_x \big[ \tilde{H}(\tfrac{y}{n}) - \tilde{H}(\tfrac{x}{n}) \big] \Big)^2\Big] \leq C(H) \begin{cases}
n^{2\theta-2\gamma} \quad &\text{if}\quad 0 < \gamma < 2,\\
n^{2\theta- 4} (\log n)^2 \quad &\text{if}\quad  \gamma = 2,\\
n^{2\theta- 4}  \quad &\text{if}\quad  \gamma > 2,\\
\end{cases} \]
which converges to zero as $n \rightarrow \infty$ by Assumption \ref{assump: parameters}.  
For the sum $\sum_{x \in \Lambda_n, y \notin \Lambda_n}$, we first rewrite it as
\[\frac{n^\theta}{\sqrt{n}} \sum_{x \in \Lambda_n} \Bar{\eta}_x H(\tfrac{x}{n}) \sum_{y \notin \Lambda_n} s_{\gamma} (y-x),\]
and then bound
\begin{align*}
	&E_{\nu_{1/2}} \Big[ \Big( \frac{n^\theta}{\sqrt{n}} \sum_{x \in \Lambda_n} \Bar{\eta}_x H(\tfrac{x}{n}) \sum_{y \notin \Lambda_n} s_{\gamma} (y-x) \Big)^2 \Big]\\
	&\leq C(\gamma)n^{2\theta-1} \sum_{x \in \Lambda_n} H(\tfrac{x}{n})^2 \big(x^{-2\gamma} + (n-x)^{-2\gamma}\big)\\
	&\leq C(\gamma)n^{2\theta-2\gamma} \int_{\tfrac{1}{n}}^{1-\tfrac{1}{n}} H(u)^2  \big(u^{-2\gamma} + (1-u)^{-2\gamma}\big) du\\
	&\leq C(\gamma,H)n^{2\theta-2\gamma} .
\end{align*}
The sum  $\sum_{x \notin \Lambda_n, y \in \Lambda_n}$ can the treated in the same way. This concludes the estimates for the term $A_s^{n,3}$. Similarly, for $A_s^{n,5}$,
\begin{align*}
	T^2 E_{\nu_{1/2}} \Big[ \Big( A^{n,5} (H)\Big)^2\Big] &\leq C n^{2\theta-1} \sum_{x \in \Lambda_n} r_{\gamma} (x)^2 H(\tfrac x n)^2\\
	&\leq C n^{2\theta-2\gamma} \int_{\tfrac 1 n}^{1-\tfrac 1 n} H(u)^2 [u^{-2\gamma} + (1-u)^{-2\gamma}] du\\
	&\leq C(H) n^{2\theta-2\gamma},
\end{align*}
thus concluding the proof.
\end{proof}

The term $A^{n,4}_s$ is the most difficult one.  The following result is proved in Section \ref{sec: pf second order boltzmann gibbs}. Define
\[\eta_{x}^{\varepsilon n} = \begin{cases}
	\frac{1}{\varepsilon n} \sum_{y=0}^{\varepsilon n -1} \eta_{x+y}, \quad &x=1,2,\ldots,n-\varepsilon n -1.\\
	\frac{1}{\varepsilon n} \sum_{y=0}^{\varepsilon n -1} \eta_{x-y}, \quad &x=n-\varepsilon n,\ldots,n -1.
\end{cases}\]

\begin{proposition}[Second order Boltzmann-Gibbs principle]\label{pro: A4}
Under assumption \ref{assump: parameters},  we have the following results:
\begin{enumerate}
	\item if $\theta < 3/2$, then
	\[\lim_{n \rightarrow \infty} \E \Big[ \sup_{0 \leq t \leq T} \Big(  \int_0^t A^{n,4}_s (H) ds \Big)^2 \Big] = 0;\]
	\item if $\theta = 3/2$, then for any $\varepsilon > 0$,
 \begin{align*}
 \limsup\limits_{n \to \infty} \mathbb{E} \Big[ \sup_{0 \leq t \leq T} \Big(  \int_0^t \big( A_s^{n,4} - m \sum\limits_{x \in \Lambda_n}   \big\{  \big( \bar{\eta}^{\varepsilon n}_x (s) \big)^2 &- \frac{1}{4\varepsilon n}  \big\} H'(\tfrac{x}{n}) \big) ds \Big)^2\Big] \\
 &\leq C T \varepsilon \|\nabla H\|_{L^2([0,1])}^2,
 \end{align*}
	where $m = \sum yp_\gamma(y)$.
\end{enumerate} 
\end{proposition}

Finally, we have  the following convergence for the martingale term, whose proof is presented in Section \ref{sec: tightness}.

\begin{proposition}\label{prop: martingale convergence}
	For any $H \in \mathcal{S}$, the sequence of martingales $\{\mathcal{M}^n_t (H), t \in [0,T]\}$ converges in distribution in $\mathcal{C} ([0,T]; \R)$, as $n \rightarrow \infty$, to the Brownian motion $\{\mathcal{M}_t (H), t \in [0,T]\}$ with variance $(t C_\alpha/2) \| \nabla H \|_{L^2([0,1])}^2$.
\end{proposition}

\subsection{Concluding the proof of Theorem \ref{thm: fluctuation}} In Section \ref{sec: tightness}, we shall prove that the sequence of processes $\{\mathcal{Y}^n_t, t \in [0,T]\}$ is tight in the space $\mathcal{C} ([0,T]; \mathcal{S}^\prime)$.  Let $\{\mathcal{Y}_t, t \in [0,T]\}$ be some limit point of the process  $\{\mathcal{Y}^n_t, t \in [0,T]\}$ along some subsequence, and we denote this converging subsequence still by $\{n\}$ for convenience. In the rest of this section, we shall show that the limit point $\{\mathcal{Y}_t, t \in [0,T]\}$ satisfies the conditions in Definition \ref{def ou 2} if $\theta < 3/2$, and in Definition  \ref{def burgers 2} if $\theta = 3/2$.  Finally, Theorem \ref{thm: fluctuation} follows  from the uniqueness of solutions to \eqref{ou} and \eqref{burgers}. 

\subsubsection{The case $\theta < 3/2$} In the last subsection, we have shown that, for any $H \in \mathcal{S}$,
\[\mathcal{M}^n_t (H) = \mathcal{Y}^n_t (H) - \mathcal{Y}^n_0 (H) - C_\alpha \int_0^t \mathcal{Y}^n_s (\Delta H) ds - \int_0^t \sum_{i = 2}^5 A^{n,i}_s (H) ds - \mathcal{E}_t^{n,1}(H),\]
	where $\mathcal{E}_t^{n,1}(H) = \int_0^t A_s^{n,1}(H) ds - C_\alpha \int_0^t \mathcal{Y}_s^n(\Delta H) ds$ and
	\[\lim_{ n \rightarrow \infty} \E \Big[ \sup_{0 \leq t \leq T} \Big(\int_0^t \sum_{i = 2}^5 A^{n,i}_s (H) ds + \mathcal{E}_t^{n,1}(H) \Big)^2\Big] = 0.\]
Letting $n \rightarrow \infty$, the limiting point satisfies that
\[\mathcal{M}_t (H) = \mathcal{Y}_t (H) - \mathcal{Y}_0 (H) - C_\alpha \int_0^t \mathcal{Y}_s (\Delta H) ds\]
is a martingale. Moreover, by Proposition \ref{prop: martingale convergence}, 
\[\mathcal{M}_t (H)^2 -  \frac{t C_\alpha}{2} \| \nabla H \|_{L^2([0,1])}^2\]
is also a martingale.  This verifies condition $(2)$ in Definition \ref{def ou 2}.  The first condition in Definition \ref{def ou 2} follows directly from the fact that the process $\eta(t)$ is stationary. At last, for condition $(3)$, it suffices to prove the following result.

\begin{lemma}\label{lem: boundary condition} For $u \in \{0,1\}$,
	\[\lim_{\varepsilon \rightarrow 0}\limsup_{n \rightarrow \infty} \E \Big[ \sup_{0 \leq t \leq T} \Big( \int_0^t \mathcal{Y}^n_s (\iota_{\varepsilon,u}) ds \Big)^2\Big] = 0.\]
\end{lemma}

\begin{proof}
We only consider the case $u=0$, and the other case can be proved in the same way. Note that
	\[\mathcal{Y}^n_s (\iota_{\varepsilon,0}) = \frac{1}{\varepsilon \sqrt{n}} \sum_{x=1}^{\varepsilon n} \bar{\eta}_x.\]
	By Cauchy-Schwarz inequality, it suffices to show that 
	\begin{align}
		\lim_{\varepsilon \rightarrow 0} \limsup_{n \rightarrow \infty} \E  \Big[ \sup_{0 \leq t \leq T} \Big( \int_0^t \frac{1}{\varepsilon \sqrt{n}} \sum_{x=1}^{\varepsilon n} (\bar{\eta}_x - \bar{\eta}_1) ds \Big)^2\Big] = 0,\label{bc 1}\\
		\lim_{n \rightarrow \infty} \E  \Big[ \sup_{0 \leq t \leq T} \Big( \int_0^t \sqrt{n} \bar{\eta}_1 ds \Big)^2\Big] = 0.\label{bc 2}
	\end{align}
	
	We first prove \eqref{bc 1}. By Kipnis-Varadhan inequality \cite[Proposition 6.1, Appendix 1]{klscaling}, the expectation in \eqref{bc 1} is bounded by
	\[20 T \sup_{f} \Big\{ \int \frac{1}{\varepsilon \sqrt{n}} \sum_{x=1}^{\varepsilon n} (\bar{\eta}_x - \bar{\eta}_1) f(\eta) d \nu_{1/2} - D_n (f) \Big\},\]
    where the supremum is taken over all densities with respect to the measure $\nu_{1/2}$, and $D_n (f)$ is the Dirichlet form of $f$ for the generator $L_n$ under $\nu_{1/2}$,
    \begin{align*}
    	D_n (f) := \langle f, -L_n f \rangle_{\nu_{1/2}} &= \frac{n^2}{2} \sum_{x,y \in \Lambda_n}  s_{\alpha} (x-y)E_{\nu_{1/2}}[(f(\eta^{x,y}) -f(\eta))^2] \\
    	&+ \frac{n^2}{4}\sum_{x \in \Lambda_n, y \notin \Lambda_n}  s_{\alpha} (x-y) E_{\nu_{1/2}}[(f(\eta^{x}) -f(\eta))^2]\\
    	&+ \frac{n^\theta}{2} \sum_{x,y \in \Lambda_n}  p_{\gamma} (y-x)  E_{\nu_{1/2}} [ \eta_x (1-\eta_y)(f(\eta^{x,y}) -f(\eta))^2] \\
    	&+ \frac{n^\theta}{4}\sum_{x \in \Lambda_n, y \leq 0}  p_{\gamma} (x-y) E_{\nu_{1/2}} [ (1-\eta_x) (f(\eta^{x}) -f(\eta))^2]\\
    	&+ \frac{n^\theta}{4}\sum_{x \in \Lambda_n, y \geq n}  p_{\gamma} (y-x) E_{\nu_{1/2}} [ \eta_x( f(\eta^{x}) -f(\eta))^2].
    \end{align*}
     We rewrite
	\begin{align*}
		\int \frac{1}{\varepsilon \sqrt{n}} \sum_{x=1}^{\varepsilon n} (\bar{\eta}_x - \bar{\eta}_1) f(\eta) d \nu_{1/2} &= \int \frac{1}{\varepsilon \sqrt{n}} \sum_{x=1}^{\varepsilon n} \sum_{y=1}^{x-1} (\bar{\eta}_{y+1} - \bar{\eta}_y) f(\eta) d \nu_{1/2}\\
		&= \int \frac{1}{2\varepsilon \sqrt{n}} \sum_{x=1}^{\varepsilon n} \sum_{y=1}^{x-1} (\bar{\eta}_{y+1} - \bar{\eta}_y) \big( f(\eta) - f(\eta^{y,y+1}) \big) d \nu_{1/2},
	\end{align*}
	where we made the transformation $\eta \mapsto \eta^{y,y+1}$ in the last identity. By Young's inequality, the last expression is bounded by
	\begin{align*}
		\int \frac{n}{2 \varepsilon} \sum_{x=1}^{\varepsilon n} \sum_{y=1}^{x-1}   \big( f(\eta) - f(\eta^{y,y+1}) \big)^2 d \nu_{1/2} + \int \frac{1}{8 \varepsilon n^2 } \sum_{x=1}^{\varepsilon n} \sum_{y=1}^{x-1} (\bar{\eta}_{y+1} - \bar{\eta}_y)^2 d \nu_{1/2}.
	\end{align*}
	Since the first term in the last expression is bounded by $D_n (f)$, and the second one is bounded by $C \varepsilon$, the proof of \eqref{bc 1} is concluded. 
	
	For \eqref{bc 2}, by Kipnis-Varadhan inequality, the expectation is bounded by
	\[20 T \sup_{f} \Big\{ \int \sqrt{n} \bar{\eta}_1 f(\eta) d \nu_{1/2} - D_n (f) \Big\}.\]
	Making the change of variable $\eta \mapsto \eta^1$, and by Young's inequality, the integral inside the above supremum equals
	\begin{align*}
		\frac{\sqrt{n}}{2} \int \bar{\eta}_1 \big( f(\eta) - f(\eta^1)\big) d \nu_{1/2} &\leq \frac{n^2}{4}   \int  \big( f(\eta) - f(\eta^1)\big)^2 d \nu_{1/2} + \frac{1}{4n}\int (\bar{\eta}_1)^2 d \nu_{1/2} \\
		&\leq D_n (f) + C/n,
	\end{align*}
	thus concluding the proof.
\end{proof}

\subsubsection{The case $\theta =3/2$} Compared with the case $\theta < 3/2$, the term $A_s^{n,4}$  does not vanish any more. Note that, for any $H \in \mathcal{S}$,
\[\lim_{n \rightarrow \infty} \frac{1}{n} \sum_{x \in \Lambda_n}  H^\prime (x/n) = 0,\]
and that 
\[\bar{\eta}_x^{\varepsilon n} (s)= n^{-1/2} \mathcal{Y}^n_s (\iota_{\varepsilon,x/n}).\]
Thus, by Proposition \ref{pro: A4},  
\begin{equation}\label{A_n^4-lim-KPZ}
	\lim_{\varepsilon \rightarrow 0} \limsup_{n \rightarrow \infty} \E \Big[ \sup_{0 \leq t \leq T} \Big(\int_0^t A^{n,4}_s - \frac{m}{n} \sum_{x \in \Lambda_n} \mathcal{Y}^n_s (\iota_{\varepsilon,x/n})^2 H^\prime (\tfrac{x}{n})ds\Big)^2\Big] = 0.
\end{equation}
Note that we can further extract a subsequence, still denoted by $\{n\}$, such that
\[\mathcal{A}_{s,t}^\varepsilon := \lim_{ n \rightarrow \infty} - \int_s^t \frac{1}{n} \sum_{x \in \Lambda_n} \mathcal{Y}^n_r (\iota_{\varepsilon,x/n})^2 H^\prime (\tfrac{x}{n})dr\]
exists in distribution.  Using Fatou's Lemma and Proposition \ref{pro: A4},
\begin{equation}
\limsup_{n \rightarrow \infty} \E \Big[ \sup_{0 \leq t \leq T} \Big(\int_s^t A^{n,4}_r dr + m \mathcal{A}_{s,t}^\varepsilon \Big)^2\Big] \leq C(H) (t-s) \varepsilon .
\end{equation}
Thus, for any $\delta < \varepsilon$,
\begin{align*}
	\E \Big[\Big(\mathcal{A}_{s,t}^\varepsilon - \mathcal{A}_{s,t}^\delta\Big)^2 \Big] &\leq 	2 \limsup_{n \rightarrow \infty}  \E \Big[\Big(\mathcal{A}_{s,t}^\varepsilon + m^{-1}\int_s^t A^{n,4}_r dr \Big)^2 \Big] \\
	&\qquad + 2 \limsup_{n \rightarrow \infty}  \E \Big[\Big(\mathcal{A}_{s,t}^\delta + m^{-1}\int_s^t A^{n,4}_r dr \Big)^2 \Big]\\
	& \leq C(H)(t-s) \varepsilon.
\end{align*}
This verifies condition $(2)$ in Definition \ref{def burgers 2}. For condition $(3)$, by \eqref{sep-Ln}, Lemmas \ref{lem: A^n_1} and \ref{lem: A^n_2}, \eqref{A_n^4-lim-KPZ} and the definition of $\mathcal{A}_{s,t}^\varepsilon$ therein, letting $n \rightarrow \infty$ in \eqref{martingale}, we have
\[\mathcal{M}_t (H) = \mathcal{Y}_t (H) - \mathcal{Y}_0 (H) - C_\alpha \int_{0}^{t} \mathcal{Y}_s (\Delta H) ds + m \mathcal{A}_{0,t}^\varepsilon (H)\]
is a martingale with quadratic variation $(tC_\alpha/2) \|\nabla H\|_{L^2 ([0,1])}^2$ for any $\varepsilon > 0$.  We conclude condition $(3)$ by letting $\varepsilon \rightarrow 0$. Conditions $(1)$ and $(4)$ can be proved in the same way as in the case $\theta < 3/2$. For condition $(5)$, observe that $\{\eta_{T-t}, t \in [0,T]\}$ is a Markov process with generator $L_n^* := n^2 L_s + n^\theta L_a^*$, where $L_a^*$ is similar to $L_a$ with $p_\gamma (\cdot)$ replaced by $p_\gamma(-\cdot)$, which is sufficient to conclude the proof.

\section{Proof of Proposition \ref{pro: A4}}\label{sec: pf second order boltzmann gibbs}

\subsection{The case $\theta < 1$}  Recall
\[A^{n,4}_s(H) = -\frac{n^\theta}{\sqrt{n}} \sum_{x,y \in \Lambda_n} a_{\gamma} (y-x) \big[ H(\tfrac{y}{n}) - H(\tfrac{x}{n}) \big] \Bar{\eta}_x\Bar{\eta}_y.\]
By first using Cauchy-Schwarz inequality, and then  developing the square, we bound
\begin{align*}
&\E \Big[ \sup_{0 \leq t \leq T} \Big(  \int_0^t A^{n,4}_s (H) ds \Big)^2 \Big] \leq T^2 E_{\nu_{1/2}} \Big[ \Big( \frac{n^\theta}{\sqrt{n}} \sum_{x,y \in \Lambda_n} a_{\gamma} (x-y) \big[ H(\tfrac{y}{n}) - H(\tfrac{x}{n}) \big] \Bar{\eta}_x\Bar{\eta}_y \Big)^2\Big]\\
&\leq \frac{T^2n^{2\theta}}{n} \sum_{x_i, y_i \in \Lambda_n,\atop i = 1,2} a_\gamma (x_1-y_1) a_\gamma (x_2-y_2) \\
&\qquad \times [H(\tfrac {y_1} n) - H (\tfrac{x_1}{n})] [H(\tfrac {y_2} n) - H (\tfrac {x_2} n)] E_{\nu_{1/2}} [\bar{\eta}_{x_1} \bar{\eta}_{x_2} \bar{\eta}_{y_1} \bar{\eta}_{y_2}].
\end{align*}
Notice that the above expectation is zero unless $x_1 = x_2, y_1 = y_2$ or $x_1 = y_2,x_2 = y_1$. Thus, the last line is bounded by
\begin{align*}
	&\frac{CT^2n^{2\theta}}{n} \sum_{x ,y \in \Lambda_n} a_\gamma (x -y)^2 [H(\tfrac {y} n) - H (\tfrac{x}{n})]^2\\
	&\leq CT^2n^{2\theta-2\gamma-1} \iint_{|u-v| > n^{-1}, |u| \leq 1, |v| \leq 1} |u-v|^{-2\gamma-2} [H(u) - H(v)]^2 du\,dv.
\end{align*}
Since $[H(u) - H(v)]^2 \leq C(H) |u-v|^2$, the last line is bounded by 
\[C(H) T^2 \times \begin{cases}
 n^{2\theta-2\gamma-1} \quad &\text{if }\; \gamma < 1/2,\\
 n^{2\theta-2} \log n \quad &\text{if }\; \gamma = 1/2,\\
 n^{2\theta -2} \quad &\text{if }\; \gamma > 1/2.
\end{cases}\]
Since $\theta < \gamma \wedge 1$, the proof is concluded.

\subsection{The case $1 \leq \theta < 3/2$}   We shall use the multi-scale analysis introduced in \cite{gonccalves2018density}, where the authors studied the density fluctuations for the exclusion process on $\Z$ with long jumps. Here, we need to deal with the boundary terms carefully since the model we considered is defined on the line segment.  Below, let $K = K(n)$ be determined later. The following lemma allows us to restrict the sum $\sum_{x,y \in \Lambda_n}$ in $A_s^{n,4}(H)$ to $\sum_{|x-y| \leq K}$.

\begin{lemma}\label{lem-out-region-K} Under assumption \ref{assump: parameters}, if $K \gg n^{\frac{2\theta - 2}{2\gamma - 1}}$, then
	\[\lim_{n \rightarrow \infty} \E \Big[\sup_{0 \leq t \leq T}   \Big( \int_0^t \frac{n^\theta}{\sqrt{n}} \sum\limits_{\substack{|x - y| > K \\ x,y\in\Lambda_n}} a_\gamma(x-y) [ H(\tfrac{y}{n}) - H(\tfrac{x}{n}) ]
	\bar{\eta}_x \bar{\eta}_y ds \Big)^2 \Big] = 0.\]
\end{lemma}

\begin{proof}
The proof is similar to the case $\theta < 1$. Thus, we only sketch it. By Cauchy-Schwarz inequality and developing the square, the expectation in the lemma is bounded by
	\begin{align*}
		& CT^2 n ^{2\theta - 1} \sum\limits_{\substack{|x - y| > K \\ x,y\in\Lambda_n}} a_\gamma(y-x)^2 [ H(\tfrac{y}{n}) - H(\tfrac{x}{n}) ]^2\\
		\leq & CT^2 n ^{2\theta - 2\gamma - 1} \iint_{|u - v|\geq \tfrac{K}{n}, |u| \leq 1, |v| \leq 1} |u - v|^{-2\gamma - 2} [ H(u) - H(v) ]^2 dudv.
	\end{align*}
	Under assumption \ref{assump: parameters}, $\gamma > \theta \geq 1$.  Since $(H(u) - H(v))^2 \leq C(H) (u - v)^2$,  the last line is bounded by
	\[ C(H)T^2  \frac{n^{2\theta - 2}}{K^{2\gamma - 1}},\]
	thus concluding the proof.
\end{proof}

It remains to deal with the term 
\[-\int_0^t \frac{n^\theta}{\sqrt{n}} \sum\limits_{\substack{1 \leq |x - y| < K \\ x,y\in\Lambda_n}} a_\gamma(y-x) [ H(\tfrac{y}{n}) - H(\tfrac{x}{n}) ]
\bar{\eta}_x \bar{\eta}_y ds. \]
By Taylor's expansion,
\[ H(\tfrac{y}{n}) - H(\tfrac{x}{n}) =  H'(\tfrac{x}{n}) \frac{y - x}{n}
+ \frac{(y-x)^2}{2n^2} H^{\prime\prime} (u^n_{x,y})\]
for some $u^n_{x,y}$ between $x/n$ and $y/n$. The following lemma shows that the second term on the right-hand side of the last identity is negligible.

\begin{lemma}\label{lem-Taylor-high-order} Under assumption \ref{assump: parameters}, we have
	\[\lim_{n \rightarrow \infty} \E \Big[  \sup_{0 \leq t \leq T} \Big( \int_0^t \frac{n^\theta}{\sqrt{n}} \sum\limits_{\substack{ 1 \leq  |x - y|\leq K \\ x,y\in\Lambda_n}} a_\gamma(x-y) 
	\bar{\eta}_x \bar{\eta}_y
	\frac{(y - x)^2}{n^2}  H^{\prime\prime} (u^n_{x,y}) ds \Big)^2 \Big] = 0.\]
\end{lemma}

\begin{proof}
By Cauchy-Schwarz inequality and developing the square, the expectation in the lemma is bounded by
	\begin{align*}
		& CT^2n^{2\theta -1} \sum\limits_{\substack{1 \leq |x - y|\leq K \\ x,y\in\Lambda_n}} a_\gamma(y-x) ^2 \frac{(y - x)^4}{n^4}
	H^{\prime\prime} (u^n_{x,y})^2\\
		\leq & \, C(H)T^2n^{2\theta - 2\gamma - 1}
		\iint_{\tfrac{1}{n} \leq |u - v| \leq \frac{K}{n}, |u| \leq 1, |v| \leq 1} |u - v|^{2 - 2\gamma} dudv\\
		\leq & \, C (H) T^2 \times \begin{cases}
		n^{2\theta - 2 \gamma -1} \quad &\text{if }\; \gamma \in (1,\frac32),\\
		n^{2\theta - 4} \log n \quad &\text{if }\; \gamma =\frac{3}{2},\\
		n^{2\theta - 4}\quad &\text{if }\; \gamma \in (\frac32, \infty).
		\end{cases}
	\end{align*}
We conclude the proof by noting that $\theta < 2 \wedge \gamma$.
\end{proof}

Next, we deal with the term 
\begin{equation}\label{eqn 2}
\int_0^t n^{\theta-3/2} \sum\limits_{\substack{ 1 \leq  |x - y| \leq K \\ x,y\in\Lambda_n}} (x-y)a_\gamma(x-y) H^\prime (\tfrac{x}{n})
\bar{\eta}_x \bar{\eta}_y ds. 
\end{equation}

We first recall the following two propositions, see \cite{gonccalves2014nonlinear} for example. For $f \in L^2(\nu_{1/2})$, the $\|\cdot\|_{-1, n}$ norm is defined by
\[ \|f\|_{-1, n}^2 = \sup\limits_{g \;\textrm{local}}
\left\{ 2\langle f, g \rangle_{\nu_{1/2}} - 
\langle g, -L_n g \rangle_{\nu_{1/2}} \right\}. \]

\begin{proposition}[Kipnis-Varadhan inequality]
For any $f: [0,T] \rightarrow L^2(\nu_{1/2})$,
	\[ \mathbb{E} \Big[ \sup\limits_{0\leq t\leq T} \Big(  \int_0^t f(s, \eta(s)) ds \Big)^2 \Big]
	\leq 14 \int_0^T \| f(t, \cdot) \|_{-1, n}^2 dt.\]
	\end{proposition}

\begin{proposition}
	
	Let $m\in\mathbb{N}$ and $k_0 < \cdots < k_m$ be an integer sequence. Suppose $\{ f_1, \cdots, f_m\}$ is a sequence of local functions such that ${\rm supp}(f_i)\subset \{k_{i-1} + 1, \cdots, k_i\}$ for any $i \in \{1, \cdots, m\}$. Set $l_i = k_i - k_{i-1}$ and suppose $\int f_i \,d\nu_{\rho} = 0$ for any $i=1, \cdots, m$ and any $\rho \in [0,1]$. Then, there  exists some $\kappa>0$ such that 
	\[ \| f_1 + \cdots + f_m \|_{-1, n}^2 
	\leq \kappa \sum\limits_{i=1}^m \frac{l_i^2}{n^2}\int f_i^2 d\nu_{1/2}. \]
\end{proposition}

Combining the last two propositions, we have:

\begin{proposition}\label{pro--1-morm-est}
	
	Let $\{ f_1, ... , f_m \}$ be as in the last Proposition. Then for $f = f_1 + \cdots + f_m$,
	\[ \mathbb{E} \Big[ \sup\limits_{0\leq t\leq T} \Big(  \int_0^t f(s, \eta(s)) ds \Big)^2 \Big]
	\leq 14 \kappa \int_0^T \sum\limits_{i=1}^m \frac{l_i^2}{n^2}\int f_i^2(s, \eta) d\nu_{1/2}ds.\]
	
\end{proposition}

In order to deal with \eqref{eqn 2}, we divide the region $\{(x,y): 1 \leq  |x-y| \leq K, x,y \in \Lambda_n\}$ into the following four parts (see Figure \ref{fig: 1}):
\begin{equation}\label{def-region}
	\begin{aligned}
		& \textbf{I} := \{ (x,y) \in \mathbb{N}^2; x = 1, ..., n - K - 1, \;
		y = x + 1, ..., x + K\},\\
		& \textbf{II} := \{ (x,y) \in \mathbb{N}^2; y = 1, ..., n - K - 1, \;
		x = y + 1, ..., y + K\},\\
		& \textbf{III} := \{ (x,y) \in \mathbb{N}^2; x = n - K, ..., n - 2, \;
		y = x + 1, ..., n - 1\},\\
		& \textbf{IV} := \{ (x,y) \in \mathbb{N}^2; y = n - K, ..., n - 2, \;
		x = y + 1, ..., n - 1\}.
	\end{aligned}
\end{equation}
We remark that the main parts are $\textbf{I}$ and $\textbf{II}$, and the boundary parts are $\textbf{III}$ and $\textbf{IV}$. We only deal with  $\textbf{I}$ and $\textbf{III}$, and the cases for $\textbf{II}$ and $\textbf{IV}$ are similar.

\begin{figure}[htbp]
\begin{tikzpicture}[scale=0.9]
	\draw[->] (-1,0) -- (4,0)node[below]{$x$};
	\draw[->] (0,-1) -- (0,4)node[left]{$y$};
	\draw (0,3) -- (3,3) -- (3,0);
	\draw (0,1) -- (2,3);
	\draw (1,0) -- (3,2);
	\draw (0,0) -- (3,3);
	\draw (2,3) -- (2,2) -- (3,2);
	\node at (1,1.5) {\small $\textbf{I}$};
	\node at (2,1.5) {\small $\textbf{II}$};
	\node at (2.3,2.7) {\small $\textbf{III}$};
	\node at (2.7,2.3) {\small $\textbf{IV}$};
\end{tikzpicture}
\caption{Four parts of the region $\{(x,y): 1 \leq  |x-y| \leq K, x,y \in \Lambda_n\}$.}
\label{fig: 1}
\end{figure}
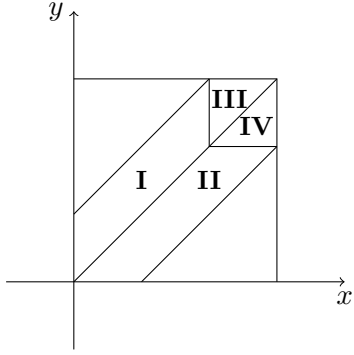

For any $2 \leq \ell \leq n-1$, $x \leq n-1-\ell$, $\eta \in \mc{X}_n$, define $\psi_x^\ell: \mc{X}_n \rightarrow \mathbb{R}$ as
\[ \psi_x^\ell (\eta) = E_{\nu_{1/2}} [\bar{\eta}(x)\bar{\eta}(x+1) | \eta^\ell (x)], \]
where 
\[ \eta^\ell ( x ) = \frac{1}{\ell} \sum\limits_{i=0}^{\ell-1} \eta( x + i). \]
Direct calculations show that
\[\psi_x^{\ell} = \frac{\ell}{\ell -1} \Big\{  \big( \bar{\eta}^{\ell}_x \big)^2 - \frac{1}{4\ell}  \Big\}.\]
Then, it is easy to see that there is some constant $C>0$ such that
\begin{equation}\label{psi-E2-est}
	\int \psi_x^\ell (\eta)^2d\nu_{1/2} \leq \frac{C}{\ell^2}.
\end{equation}

We first consider the situation when $(x, y) \in \textbf{I}$.  Then, we can rewrite
\begin{align*}
	\int_0^t n^{\theta-3/2} \sum\limits_{(x,y) \in \mathbf{I}} (x-y)a_\gamma(x-y) H^\prime (\tfrac{x}{n})
	\bar{\eta}_x \bar{\eta}_y ds = 	\int_0^t n^{\theta-3/2}  \sum_{x=1}^{n-1-K} \sum_{y=1}^K ya_\gamma(y) H^\prime (\tfrac{x}{n})
	\bar{\eta}_x \bar{\eta}_{x+y} ds .
\end{align*}
For the convenience of calculation, suppose $(n - 1)/K \in \mathbb{N}$ in the following.  For $j=1,2,\ldots,K$, define 
\begin{equation}\label{def-H-j}
	\mathcal{H}_j = \{ K z + j; z = 0, 1,..., \tfrac{n-1}{K} - 2\}. 
\end{equation}
Then, $\{1,2,\ldots,n-K-1\} = \cup_{j=1}^K \mathcal{H}_j$. By Cauchy-Schwarz inequality,
\begin{align}
	&\E \Big[ \sup_{0 \leq t \leq T} \Big(	\int_0^t n^{\theta-3/2}  \sum_{x=1}^{n-1-K} \sum_{y=1}^K ya_\gamma(y) H^\prime (\tfrac{x}{n})
	\bar{\eta}_x \bar{\eta}_{x+y} ds \Big)^2\Big] \notag \\
	&\leq 2\mathbb{E} \Big[\sup\limits_{0 \leq t \leq T}  \Big(  \int_0^t n^{\theta - \frac32} \sum\limits_{\substack{x\in\mathcal{H}_j \\ j = 1,...,K}}
	\sum\limits_{y = 1}^{K} y a_\gamma(y) \{ \bar{\eta}_x \bar{\eta}_{x+y} - \psi_x^{K}{ (\eta(s))}\} H'(\tfrac{x}{n})ds\Big)^2\Big]\label{eqn 3}\\
	&\quad+ 2\mathbb{E} \Big[ \sup\limits_{0 \leq t \leq T} \Big(  \int_0^t n^{\theta - \frac32} \sum\limits_{\substack{x\in\mathcal{H}_j \\ j = 1,...,K}}
	\sum\limits_{y = 1}^{K} y a_\gamma(y) \psi_x^{K}{ (\eta(s))} H'(\tfrac{x}{n})ds\Big)^2\Big].\label{eqn 4}
\end{align}

We first bound the term \eqref{eqn 4}.

\begin{lemma}\label{lem-one-order-psi}
There is some constant $C = C(H) > 0$ such that
	\[ \mathbb{E} \Big[\sup\limits_{0 \leq t \leq T}  \Big(  \int_0^t n^{\theta - \frac32} \sum\limits_{\substack{x\in\mathcal{H}_j \\ j = 1,...,K}}
	\sum\limits_{y = 1}^{K} y a_\gamma(y) \psi_x^{K}( \eta(s)) H'(\tfrac{x}{n})ds\Big)^2\Big] \leq CT^2\frac{n^{2\theta-2}}{K}. \]
\end{lemma}

\begin{proof}
	
	By Cauchy-Schwarz inequality, since $\gamma > 1$, the expectation in the lemma is bounded by
	\begin{align*}
		&T^2 n^{2\theta - 3} \Big( \sum\limits_{y = 1}^{K} y a_\gamma(y) \Big)^2 E_{\nu_{1/2}} \Big[\Big( \sum\limits_{\substack{x\in\mathcal{H}_j \\ j = 1,...,K}} 
		\psi_x^{K}( \eta ) H'(\tfrac{x}{n}) \Big)^2\Big] \\
		\leq & C T^2 n^{2\theta-3} K\sum\limits_{j=1}^{K} E_{\nu_{1/2}} \Big[\Big( \sum\limits_{x\in\mathcal{H}_j} \psi_x^{K}( \eta ) H'(\tfrac{x}{n}) \Big)^2\Big]\\
		\leq  & C T^2 n^{2\theta-3} K\sum\limits_{j=1}^{K} \sum\limits_{x\in\mathcal{H}_j} H'(\tfrac{x}{n})^2
		E_{\nu_{1/2}} \Big[\psi_x^K( \eta )^2\Big].
	\end{align*}
In the last inequality, we used the fact that $\psi_x^K$ and $\psi_y^K$ are independent for $x \neq y \in \mathcal{H}_j$. By \eqref{psi-E2-est}, the last  expression is  bounded by $	CT^2n^{2\theta-2} / K$,
	which ends the proof.
\end{proof}

Next, we bound the term \eqref{eqn 3}. 

\begin{lemma}\label{lem-Ltheta-differ} There exists some constant $C = C(H) > 0$ such that
	\begin{equation}\label{Ltheta-differ}
		\begin{aligned}
			&\mathbb{E} \Big[ \sup\limits_{0 \leq t \leq T} \Big(  \int_0^t n^{\theta - \frac32} \sum\limits_{x = 1}^{n - K - 1} \sum\limits_{y = 1}^K y a_\gamma(y) \{ \bar{\eta}_x \bar{\eta}_{x+y} - \psi_x^{K}( \eta (s))\} H'(\tfrac{x}{n})ds\Big)^2\Big]\\
			& \leq \, C T \left( \frac{K^{4 - 2\gamma}}{n^{4 - 2\theta}}
			+ \frac{K}{n^{4 - 2\theta}} \right).
		\end{aligned}
	\end{equation}
\end{lemma}

\begin{proof}
Without loss of generality, suppose that there are some $l \in \mathbb{N}$ and $k \in \mathbb{N}$ such that $K = 2^{l} k$. Set $K_i = 2^{i}k, i = 1,... ,l$, then $ 0 =: K_0 < K_1 < \cdots < K_l = K$. By Cauchy-Schwarz inequality,  the expectation in the lemma  is bounded by
	\begin{align}\label{est-differ}
		& 2\mathbb{E} \Big[ \sup\limits_{0 \leq t \leq T} \Big( \int_0^t n^{\theta - \frac32} \sum\limits_{x = 1}^{n - K - 1} \sum\limits_{i = 1}^l \sum\limits_{y = K_{i-1} + 1}^{K_i} y a_\gamma(y) \{ \bar{\eta}_x \bar{\eta}_{x+y} - \psi_x^{K_i}( \eta (s))\} H'(\tfrac{x}{n})ds\Big)^2\Big]\notag\\
		& \quad  + 2\mathbb{E} \Big[ \sup\limits_{0 \leq t \leq T} \Big(  \int_0^t n^{\theta - \frac32} \sum\limits_{x = 1}^{n - K - 1}  \sum\limits_{i = 1}^l \sum\limits_{y = K_{i-1} + 1}^{K_i} y a_\gamma(y) \{ \psi_x^{K_i}( \eta (s)) - \psi_x^{K}( \eta (s))\} H'(\tfrac{x}{n})ds\Big)^2\Big].
	\end{align}
	
Using the Minkowski's inequality, the first expectation is bounded by
	\begin{equation}\label{est-differ-1}
			\left( \sum\limits_{i = 1}^l \bigg( \mathbb{E} \Big[ \sup\limits_{0 \leq t \leq T} \Big(  \int_0^t n^{\theta - \frac32} \sum\limits_{x = 1}^{n - K - 1} \sum\limits_{y = K_{i-1} + 1}^{K_i} y a_\gamma(y) \{ \bar{\eta}_x \bar{\eta}_{x+y} - \psi_x^{K_i}( \eta (s))\} H'(\tfrac{x}{n})ds\Big)^2\Big] \bigg)^{1/2} \right)^2.
 \end{equation}
	Since $K = K_l = 2^{l - i} K_i$, we can rewrite the sum $\sum_{x = 1}^{n - K -1}$ as $\sum_{j = 1}^{K_i}\sum_{x \in \mathcal{H}_j^i}$, where
	\begin{equation}\label{def-H-ij}
		\mathcal{H}_j^i = \{ K_i z + j; z = 0, 1, ..., \tfrac{n-1}{K_i} - (2^{l - i} + 1)\}, \quad j = 1, ..., K_i. 
	\end{equation}
	By Cauchy-Schwarz inequality and Proposition \ref{pro--1-morm-est} , we can further bound the expectation in \eqref{est-differ-1} by
	\begin{align}\label{est-A4-differ}
		& K_i n^{2\theta - 3} \sum\limits_{j = 1}^{K_i} \mathbb{E} \Big[  \sup\limits_{0 \leq t \leq T} \Big( \int_0^t  \sum\limits_{x \in \mathcal{H}_j^i} \sum\limits_{y = K_{i-1} + 1}^{K_i} y a_\gamma(y) \{ \bar{\eta}_x \bar{\eta}_{x+y} - \psi_x^{K_i}( \eta (s))\} H'(\tfrac{x}{n})ds\Big)^2\Big]\notag\\
		\leq & C T K_i n^{2\theta - 3} \sum\limits_{j = 1}^{K_i}
	 \sum\limits_{x \in \mathcal{H}_j^i} \frac{K_i^2}{n^2} H'(\tfrac{x}{n})^2 E_{\nu_{1/2}}\Big[\Big( \sum\limits_{y = K_{i-1} + 1}^{K_i} y a_\gamma(y) \{\bar{\eta}_x \bar{\eta}_{x + y} - \psi_x^{K_i}(\eta)\}\Big)^2 \Big].
	\end{align}
	Using \eqref{psi-E2-est}, 
	\begin{align*}
		 E_{\nu_{1/2}}\Big[\Big( \sum\limits_{y = K_{i-1} + 1}^{K_i} y a_\gamma(y) \{\bar{\eta}_x \bar{\eta}_{x + y} - \psi_x^{K_i}(\eta)\}\Big)^2 \Big]
		\leq & 2 \sum\limits_{y = K_{i - 1} + 1}^{K_i} y^2 a_\gamma(y)^2 + \frac{2}{K_i^2} \Big( \sum\limits_{y = K_{i - 1} + 1}^{K_i} y a_\gamma(y) \Big)^2\\
		\leq & 2 K_{i - 1}^{1 - 2\gamma} + \frac{2K_{i - 1}^{2 - 2\gamma}}{K_i^2} \leq C K_{i - 1}^{1 - 2\gamma}.
	\end{align*}
Thus,
	\begin{align*}
		\eqref{est-A4-differ} \leq  C T K_i n^{2\theta - 3} \sum\limits_{x = 1}^{n - K -1} H'(\tfrac{x}{n})^2 \frac{K_i^2}{n^2} K_{i - 1}^{1 - 2\gamma}
		\leq  \frac{C T K_i^3 n^{2\theta - 4}}{K_{i - 1}^{2\gamma - 1}}.
	\end{align*}
Finally, by \eqref{est-differ-1}, the first expectation in \eqref{est-differ} is bounded by 
	\begin{align*}
\left( \sum\limits_{i=1}^l \sqrt{\frac{C T K_i^3 n^{2\theta - 4} }{K_{i - 1}^{2\gamma - 1}}} \right)^2
		 \leq C T n^{2\theta - 4} \left( \sum\limits_{i=1}^l 2^{(2 - \gamma)i} \right)^2
		\leq C T n^{2\theta - 4} 2^{(4 - 2\gamma)l}
		 \leq C T \frac{K^{4 - 2\gamma}}{n^{4 - 2\theta}}.
	\end{align*}

	Now we bound the second expectation in \eqref{est-differ}. We first rewrite the sum inside the expectation as
	\begin{align*}
	&\sum\limits_{x = 1}^{n - K - 1}  \sum\limits_{i = 1}^l \sum\limits_{y = K_{i-1} + 1}^{K_i} y a_\gamma(y) \sum\limits_{j = i}^{l -1}\{ \psi_x^{K_j} - \psi_x^{K_{j + 1}}\} H'(\tfrac{x}{n}) \\
	&= \sum\limits_{x = 1}^{n - K - 1} \sum\limits_{j = 1}^{l -1} \sum\limits_{i = 1}^j \sum\limits_{y = K_{i-1} + 1}^{K_i} y a_\gamma(y) \{ \psi_x^{K_j} - \psi_x^{K_{j + 1}}\} H'(\tfrac{x}{n})\\
	&= \sum\limits_{x = 1}^{n - K - 1} H'(\tfrac{x}{n}) \sum\limits_{j = 1}^{l -1} \{ \psi_x^{K_j} - \psi_x^{K_{j + 1}}\} \sum\limits_{y = 1}^{K_j} y a_\gamma(y).
	\end{align*}
	Using the Minkowski's inequality and noting that $\sum_{y = 1}^{K_j} y a_\gamma(y)$ is uniformly bounded for 
	$j = 1,...,l-1$, the second expectation in \eqref{est-differ}  is bounded by
	\begin{equation}\label{est-differ-2}
		C \left( \sum\limits_{j = 1}^{l - 1} \bigg( \E \Big[ \sup\limits_{0 \leq t \leq T} \Big(  \int_0^t n^{\theta - \frac32} \sum\limits_{x = 1}^{n - K - 1} H'(\tfrac{x}{n}) \{ \psi_x^{K_j} - \psi_x^{K_{j + 1}} \}   ds\Big)^2\Big] \bigg)^{\frac12} \right)^2. 
	\end{equation}
	Rewrite the sum $\sum_{x = 1}^{n - K -1}$ as $\sum_{i = 1}^{K_{j + 1}}\sum_{x \in \mathcal{H}_i^{j + 1}}$, where $\mathcal{H}_i^{j + 1}$ is defined in \eqref{def-H-ij}. Then, by Cauchy-Schwarz inequality,  Proposition \ref{pro--1-morm-est} and \eqref{psi-E2-est},
	\begin{align*}
		& \mathbb{E} \Big[ \sup\limits_{0 \leq t \leq T} \Big(  \int_0^t n^{\theta - \frac32} \sum\limits_{i = 1}^{K_{j + 1}} \sum\limits_{x \in \mathcal{H}_i^{j + 1}} H'(\tfrac{x}{n}) \{ \psi_x^{K_j} - \psi_x^{K_{j + 1}} \}   ds\Big)^2\Big]\\
		\leq & C T n^{2\theta - 3} K_{j + 1} \sum\limits_{i = 1}^{K_{j+1}}   \sum\limits_{x \in \mathcal{H}_i^{j+1}} \frac{K_{j+1}^2}{n^2}  H'(\tfrac{x}{n})^2 E_{\nu_{1/2}}\Big[ \{\psi_x^{K_j} - \psi_x^{K_{j + 1}}\}^2\Big] \\
		\leq & C T  n^{2\theta - 3}  \frac{K_{j+1}^3}{n^2}  \sum\limits_{x = 1}^{n - K - 1}   H'(\tfrac{x}{n})^2 \frac{1}{K_j^2} \\
		\leq & C T \frac{K_{j + 1}}{n^{4 - 2\theta}}.
	\end{align*}
Thus, we bound \eqref{est-differ-2}  by
	\[   C T \left( \sum\limits_{j = 1}^{l - 1} \sqrt{\frac{K_{j + 1}}{n^{4 - 2\theta}}} \right)^2 
\leq   \frac{C T K}{n^{4 - 2\theta}}. \]
This concludes the proof.
\end{proof}

\medspace

Next, we deal with the situation when $(x, y) \in \textbf{III}$.  We have the following bound

\begin{lemma}\label{lem: III} For any $\lambda \in \mathbb{N}_+$, there exists some constant $C=C(\lambda,H)$ such that
	\begin{equation}\label{est-l-III}
		\E \Big[  \sup_{0 \leq t \leq T} \Big( \int_0^t n^{\theta - \frac32} \sum\limits_{x = n - K}^{n - 2} \sum\limits_{y = 1}^{n - 1 - x} y a_\gamma(y) \bar{\eta}_x \bar{\eta}_{x+y}
		H'(\tfrac{x}{n}) ds \Big)^2 \Big] 	\leq C \Big( T \frac{K^{\lambda + 1}}{n^{2 - 2\theta + \lambda}} + T^2 \frac{K^{\lambda + 1}}{n^{4 - 2\theta + \lambda}} \Big). 
	\end{equation}
\end{lemma}

By Cauchy-Schwarz inequality, the expectation in the above lemma can be bounded by
\begin{align*}
	& 2 \E  \Big[ \sup_{0 \leq t \leq T} \Big(  \int_0^t n^{\theta - \frac32} \sum\limits_{x = n - K}^{n - 2} \sum\limits_{y = 1}^{n - 1 - x} y a_\gamma(y) \{\bar{\eta}_x \bar{\eta}_{x+y} - \psi_x^{K_x} (\eta(s)) \}
	H'(\tfrac{x}{n}) ds \Big)^2 \Big]\\
	&\quad  + 2 \E \Big[  \sup_{0 \leq t \leq T} \Big( \int_0^t n^{\theta - \frac32} \sum\limits_{x = n - K}^{n - 2} \sum\limits_{y = 1}^{n - 1 - x} y a_\gamma(y) \psi_x^{K_x} (\eta(s))
	H'(\tfrac{x}{n}) ds \Big)^2 \Big],
\end{align*}
where $K_x = n - 1 - x$. Thus, Lemma \ref{lem: III} follows directly from the following two lemmas.

\begin{lemma}\label{lem-III-1} For any $\lambda \in \mathbb{N}_+$, there exists some constant $C=C(\lambda,H)$ such that
	\[ \E \Big[ \sup_{0 \leq t \leq T} \Big(  \int_0^t n^{\theta - \frac32} \sum\limits_{x = n - K}^{n - 2} \sum\limits_{y = 1}^{n - 1 - x} y a_\gamma(y) \{\bar{\eta}_x \bar{\eta}_{x+y} - \psi_x^{K_x} (\eta(s)) \}
	H'(\tfrac{x}{n}) ds \Big)^2 \Big] \leq C T \frac{K^{\lambda + 1}}{n^{2 - 2\theta + \lambda}}. \]
\end{lemma}

\begin{proof}
	
	Define 
	\[ F_x (\eta) = \sum\limits_{y = 1}^{n - 1 - x} y a_\gamma(y) \{ \bar{\eta}_x \bar{\eta}_{x + y} - \psi_x^{K_x}(\eta) \}. \]
	Then, the expectation can be bounded as
	\[ K n^{2\theta - 3} \sum\limits_{x = n - K}^{n - 2} H'(\tfrac{x}{n})^2 \E \Big[ \sup_{0 \leq t \leq T} \Big(  \int_0^t F_x( \eta(s) ) ds \Big)^2 \Big]. \]
	Using Proposition \ref{pro--1-morm-est} to $f = F_x$,  the last line is bounded by
	\begin{align*}
C T K n^{2\theta - 3} \sum\limits_{x = n - K}^{n - 2} H'(\tfrac{x}{n})^2 \frac{K_x^2}{n^2} \sum\limits_{y = 1}^{n - 1 - x} y^2 a_\gamma(y)^2.
	\end{align*}
	Since $\sum_{y = 1}^{n - 1 - x} y^2 a_\gamma(y)^2$ is uniformly bounded for $x = n - K, ..., n - 2$, the last line is bounded by
	\[ C T K n^{2\theta - 2} \int_{1 - \frac{K}{n}}^{1 - \frac{2}{n}} H'(u)^2 (1 - u)^2 du. \]
	Since $H \in \mathcal{S}$, for any $\frac{\lambda - 3}{2} \in \mathbb{N}_+$, $H'(u) \leq C(\lambda,H) (1 - u)^{\frac{\lambda - 3}{2}}$. So the last line can further be bounded by
	\[ C(\lambda,H) T K n^{2\theta - 2} \int_{1 - \frac{K}{n}}^{1 - \frac{2}{n}}
	(1 - u)^{\lambda - 1} du
	\leq  C(\lambda,H) T \frac{K^{\lambda + 1}}{n^{2 - 2\theta + \lambda}}, \]
thus concluding the proof.
\end{proof}

\begin{lemma}\label{lem-III-2}  For any $\lambda \in \mathbb{N}_+$, there exists some constant $C=C(\lambda,H)$ such that
	\[ \E \Big[ \sup_{0 \leq t \leq T} \Big(  \int_0^t n^{\theta - \frac32} \sum\limits_{x = n - K}^{n - 2} \sum\limits_{y = 1}^{n - 1 - x} y a_\gamma(y) \psi_x^{K_x} (\eta(s))
	H'(\tfrac{x}{n}) ds \Big)^2 \Big] \leq C T^2 \frac{K^{\lambda + 1}}{n^{4 - 2\theta + \lambda}}.
	\]
\end{lemma}

\begin{proof}
	Using the Cauchy-Schwarz inequality and \eqref{psi-E2-est}, the expectation can be bounded by
	\begin{align*}
		& T^2 n^{2\theta - 3} E_{\nu_{1/2}}  \Big[ \Big( \sum\limits_{x = n - K}^{n - 2} \sum\limits_{y = 1}^{n - 1 - x} y a_\gamma(y) \psi_x^{K_x}(\eta)  H'(\tfrac{x}{n}) \Big)^2 \Big]\\
		\leq & C T^2 K n^{2\theta - 3} \sum\limits_{x = n - K}^{n - 2} H'(\tfrac{x}{n})^2  E_{\nu_{1/2}} \Big[ \Big( \psi_x^{K_x}(\eta)  \Big)^2 \Big] \Big(\sum\limits_{y = 1}^{n - 1 - x} y a_\gamma(y) \Big)^2\\
		\leq & C T^2 K n^{2\theta - 3} \sum\limits_{x = n - K}^{n - 2} H'(\tfrac{x}{n})^2 \frac{1}{K_x^2}.
	\end{align*}
	Using the similar argument in the proof of last lemma, the last line is bounded by
	\[ C(\lambda,H) T^2 K n^{2\theta - 4} \int_{1 - \frac{K}{n}}^{1 - \frac{2}{n}} 
	(1 - u)^{\lambda - 1} du
	\leq C T^2 \frac{K^{\lambda + 1}}{n^{4 - 2\theta + \lambda}},\]
	which concludes the proof.
\end{proof}

Adding up the above estimates, we have shown that, if $K \gg n^{\tfrac{2\theta-2}{2\gamma-1}}$, then the expectation in Proposition \ref{pro: A4} is bounded by, for any $\lambda \in \N_+$,
\begin{equation}\label{bound sub}
C \Big\{ T^2 \Big(\frac{n^{2\theta-2}}{K} + \frac{K^{\lambda+1}}{n^{4-2\theta+\lambda}}\Big) + T \Big( \frac{K^{4-2\gamma}}{n^{4-2\theta}} + \frac{K}{n^{4-2\theta}} + \frac{K^{\lambda+1}}{n^{2-2\theta+\lambda}} \Big)\Big\}+o_n (1).
\end{equation}
Since $\theta < 3/2$ and $\gamma> \theta \geq 1$,
\[\frac{2\theta-2}{2\gamma-1} \leq 2\theta-2 < 1.\]
Take $K = n^{1-\delta}$ with $0 < \delta < 3 - 2 \theta$, and take $\lambda > (2-\delta-1)/\delta$.  Then, one can check that
\[2\theta-2 < 1 - \delta, \quad (1-\delta)(\lambda+1) < 2-2\theta+\lambda.\]
Thus, all the  terms in \eqref{bound sub} have order $o_n (1)$, which concludes the proof in the case $\theta < 3/2$.

\subsection{The case $\theta = \frac32$}\label{subsec: critical} In the last section, if we take $\delta$ and $\lambda$ satisfying 
\[1 > 1-\delta > \frac{1}{2\gamma-1}, \quad (1-\delta)(\lambda+1) < 2-2\theta+\lambda,\]
then we have shown that,  for $K =  n^{1 - \delta}$,
\begin{equation}\label{A4-3/2-equi}
	\lim_{ n \rightarrow \infty} \E \Big[ \sup_{0 \leq t \leq T} \Big(\int_0^t \big\{A_s^{n,4} (H) - 2 \sum_{x=1}^{n-K-1}
	\sum\limits_{y = 1}^{K} y a_\gamma(y) \psi_x^{K}( \eta(s)) H'(\tfrac{x}{n}) \big\}ds\Big)^2\Big] = 0.
\end{equation}
The factor $2$ comes from the symmetry between the two regions $\mathbf{I}$ and $\mathbf{II}$.   Note that
\[2 \sum_{y=1}^\infty  y a_\gamma (y) = \sum_{y \in \Z} y p_\gamma (y) = m.\]
Thus,
\begin{equation}\label{eqn 5}
\begin{aligned}
	&m \sum_{x=1}^{n-1}     \psi_x^{\varepsilon n}H'(\tfrac{x}{n}) - 	2 \sum_{x=1}^{n-K-1}
	\sum\limits_{y = 1}^{K} y a_\gamma(y) \psi_x^{K}H'(\tfrac{x}{n}) \\
	&=  m \sum_{x=1}^{n-\varepsilon n} ( \psi_x^{\varepsilon n} - \psi_x^K) H'(\tfrac{x}{n}) + m \sum_{x=n-\varepsilon n+1}^{n-1} \psi_x^{\varepsilon n} H'(\tfrac{x}{n}) - m \sum_{x= n-\varepsilon n+1}^{n-K-1} \psi_x^{K} H'(\tfrac{x}{n}) \\
	&\quad + 2 \sum\limits_{y = K}^{\infty} y a_\gamma(y) \sum_{x=1}^{n-K-1}  \psi_x^{K}H'(\tfrac{x}{n}) 
\end{aligned}
\end{equation}
Note that for $x=n-\varepsilon n, \ldots,n-1$,
\[\psi_x^{\varepsilon n} = E_{\nu_{1/2}} \Big[\bar{\eta}_x \bar{\eta}_{x-1} \Big| \sum^{y=x}_{x-\varepsilon n+1} \eta_y \Big] =: \psi_{x,L}^{\varepsilon n} .\]
Similarly, we denote the average over the right box by $\psi_{x,R}^{\varepsilon n}$. Then, the second and the third terms on \eqref{eqn 5} equals
\begin{equation}\label{eqn 6}
	\begin{aligned}
			&m \sum_{x=n-\varepsilon n+1}^{n-1} \psi_{x,L}^{\varepsilon n} H'(\tfrac{x}{n}) - m \sum_{x= n-\varepsilon n+1}^{n-K-1} \psi_{x,R}^{K} H'(\tfrac{x}{n}) 
			= m \sum_{x=n-\varepsilon n+1}^{n-K-1} \big(\psi_{x,L}^{\varepsilon n} - \psi_{x,L}^{K}\big) H'(\tfrac{x}{n})  \\
			&\quad + m \sum_{x=n-K}^{n-1} \psi_{x,L}^{\varepsilon n} H'(\tfrac{x}{n}) - m \sum_{x= n-\varepsilon n+1}^{n-K-1} \big(\psi_{x,R}^{K} - \psi_{x,L}^{K}\big) H'(\tfrac{x}{n}) 
	\end{aligned}
\end{equation}
The first term on the right-hand side of \eqref{eqn 6} can be dealt with in the same way as in Lemma \ref{lem 1} below.  For the second term, 
\begin{align*}
	 E_{\nu_{1/2}} \Big[\Big( \sum\limits_{x = n -K}^{n - 1} \psi_{x,L}^{\varepsilon n}(\eta) H'(\tfrac{x}{n})\Big)^2\Big]
	&\leq K \sum\limits_{x = n -K}^{n - 1}
	E_{\nu_{1/2}} [\psi_{x,L}^{\varepsilon n}(\eta) ^2 ] H'(\tfrac{x}{n})^2\\
	&\leq  C K \sum\limits_{x = n - K}^{n - 1} \frac{1}{(\varepsilon n)^2} H'(\tfrac{x}{n})^2  \leq \frac{C (H) }{\varepsilon^2 n^{\delta}}.
\end{align*}
For the last term on the right-hand side of \eqref{eqn 6}, by Cauchy-Schwarz inequality and Proposition \ref{pro--1-morm-est}, 
\begin{align*}
	&\E \Big[\sup_{0 \leq t \leq T} \Big(\int_0^t \sum_{x= n-\varepsilon n+1}^{n-K-1} \big(\psi_{x,R}^{K} - \psi_{x,L}^{K}\big) H'(\tfrac{x}{n})  ds\Big)^2\Big] \\ 
	&\leq CT \varepsilon n \sum_{x= n-\varepsilon n+1}^{n-K-1} \frac{K^2}{n^2} E_{\nu_{1/2}} \Big[\Big(\psi_{x,R}^{K} - \psi_{x,L}^{K}\Big)^2\Big] H^\prime (\tfrac{x}{n})^2 \\
	&\leq C(H) T \varepsilon^2. 
\end{align*}
For the last term on the right-hand side of \eqref{eqn 5}, by Lemma \ref{lem-one-order-psi}, 
\begin{align*}
	\E \Big[\sup_{0 \leq t \leq T} \Big(\int_0^t \sum\limits_{y = K}^{\infty} y a_\gamma(y) \sum_{x=1}^{n-K-1}  \psi_x^{K}H'(\tfrac{x}{n})  ds\Big)^2\Big] \leq \frac{CT^2n}{K^{2\gamma-1}},
\end{align*} 
which converges to zero as $n \rightarrow \infty$ since $(1-\delta) (2\gamma-1)  > 1$.

It remains to bound the first term on the right-hand side of \eqref{eqn 5}.

\begin{lemma}\label{lem 1} There exists some constant $C(H) = C \|\nabla H\|_{L^2([0,1])}^2$ for some $C>0$ such that
\[ \mathbb{E} \Big[\sup\limits_{0 \leq t \leq T}  \Big( \int_0^t \sum\limits_{x = 1}^{n - \varepsilon n} (\psi_x^{K}(\eta (s)) - \psi_x^{\epsilon n}(\eta (s))) H'(\tfrac{x}{n}) ds \Big)^2 \Big] \leq C(H) T \varepsilon. \]
\end{lemma}

\begin{proof}
	
	Using Proposition \ref{pro--1-morm-est} and \eqref{psi-E2-est}, for $M \in \mathbb{N}_+$,
	\begin{align*}
		& \mathbb{E}  \Big[\sup\limits_{0 \leq t \leq T}  \Big(  \int_0^t \sum\limits_{x = 1}^{n - \varepsilon n} (\psi_x^M(\eta (s)) - \psi_x^{2M}(\eta (s))) H'(\tfrac{x}{n}) ds \Big)^2\Big]\\
		= & \mathbb{E} \Big[ \sup\limits_{0 \leq t \leq T} \Big(  \int_0^t \sum\limits_{j = 1}^{2M} \sum\limits_{x \in \mathcal{H}_j} (\psi_x^M(\eta (s)) - \psi_x^{2M}(\eta (s))) H'(\tfrac{x}{n}) ds \Big)^2\Big]\\
		\leq & C T M \sum\limits_{j = 1}^{2M}\frac{M^2}{n^2} \sum\limits_{x \in \mathcal{H}_j} E_{\nu_{1/2}}\Big[ \Big( (\psi_x^M(\eta (s)) - \psi_x^{2M}(\eta (s))) \Big)^2\Big] H'(\tfrac{x}{n})^2 \\
		\leq & C T \frac{M^3}{n^2} \sum\limits_{x=1}^{n-\varepsilon n} \frac{1}{M^2} H'(\tfrac{x}{n})^2 \leq  C(H) T \frac{M}{n},
	\end{align*}
	where $\mathcal{H}_j$ is the same with \eqref{def-H-j} except that $K$ is replaced with $2M$.
	
Without loss of generality, let us assume $\varepsilon n = 2^l K$ for some integer $l$. Set $K_0 = K$ and $K_i = 2^i K$ for $i \in \mathbb{N}$. Then, using  Minkowski's inequality and the estimate above, we have
	\begin{align*}
		& \mathbb{E} \Big[\sup\limits_{0 \leq t \leq T} \Big(  \int_0^t \sum\limits_{x=1}^{n-\varepsilon n} (\psi_x^K (\eta (s)) - \psi_x^{\varepsilon n} (\eta (s))) H'(\tfrac{x}{n}) ds \Big)^2 \Big]\\
		= & \mathbb{E} \Big[\sup\limits_{0 \leq t \leq T} \Big(  \int_0^t \sum\limits_{x=1}^{n-\varepsilon n} \sum\limits_{i = 0}^{l - 1} (\psi_x^{K_i}(\eta (s)) - \psi_x^{K_{i+1}}(\eta (s))) H'(\tfrac{x}{n}) ds \Big)^2 \Big]\\
		\leq & \left( \sum\limits_{i = 0}^{l - 1} \Big( \mathbb{E} \Big[\sup\limits_{0 \leq t \leq T} \Big(  \int_0^t \sum\limits_{x=1}^{n-\varepsilon n} (\psi_x^{K_i}(\eta (s)) - \psi_x^{K_{i+1}}(\eta (s))) H'(\tfrac{x}{n}) ds \Big)^2 \Big] \Big)^{1/2} \right)^2\\
		\leq & \bigg( \sum\limits_{i = 0}^{l - 1} \sqrt{C(H) T \frac{K_i}{n}} \bigg)^2 \leq C(H) T \varepsilon.
	\end{align*}
This  ends the proof.
\end{proof}

\section{Tightness}\label{sec: tightness}

In this section, we first prove that the sequence $\{\mathcal{Y}_t^n, t \in [0,T]\}_{n \geq 1}$ is tight in  the space $\mathcal{D}([0, T]; \mathcal{S}')$ equipped with the $J_1$-{Skorokhod} topology,  then we prove that any limiting point concentrates on the continuous trajectory.

We first recall Mitoma's criterion \cite{mitoma1983tightness}.

\begin{proposition}[Mitoma's criterion]\label{pro-Mitoma}
A sequence $\{ \mathcal{Y}_t^n, t\in[0,T]\}_{n\in\mathbb{N}}$ of stochastic processes in $\mathcal{D}([0,T]; \mathcal{S}')$ is tight with respect to the $J_1$-{Skorokhod} topology if and only if the sequence of real-valued processes $\{\mathcal{Y}_t^n(H), t\in[0,T]\}_{n\in\mathbb{N}}$ is tight with respect to the $J_1$-{Skorokhod} topology of $\mathcal{D}([0,T]; \mathbb{R})$ for any $H \in\mathcal{S}$.
\end{proposition}

By \eqref{martingale}, \eqref{sep-Ln} and  Proposition \ref{pro-Mitoma}, we only need to prove the tightness of the following processes
\[ \left\{ \mathcal{Y}_0^n(H) \right\}_{n \in \mathbb{N}}, \quad \left\{ \mathcal{M}_t^n(H), t \in [0,T] \right\}_{n \in \mathbb{N}}, \quad \left\{ \int_0^t A_s^{n,j}(H)ds, t \in [0,T] \right\}_{n \in \mathbb{N}},\; j= 1,...,5 \]
for any $H \in \mathcal{S}$, respectively.

\subsection{Convergence of the initial fluctuation field}
In this subsection, we show that the sequence $\{\mathcal{Y}_0^n(H)\}_{n \in \mathbb{N}}$ converges in distribution to a Gaussian random variable $\mathcal{Y}_0(H)$ of mean zero and variance $\frac14 \| H \|_{L^2([0,1])}^2$. In particular, $\{\mathcal{Y}_0^n(H)\}_{n \in \mathbb{N}}$ is tight.

Indeed, by Taylor's expansion, for any $\lambda \in \R$,
\begin{align*}
\lim_{n \rightarrow \infty} \log \mathbb{E} \Big[\exp\{ i \lambda \mathcal{Y}_0^n(H) \} \Big] & =\lim_{n \rightarrow \infty} \log \int \exp \left\{ \frac{i \lambda}{\sqrt{n}} \sum\limits_{x \in \Lambda_n} \left(\eta_x - \tfrac12 \right) H(\tfrac{x}{n}) \right\} d\nu_{1/2}\\
& = \lim_{n \rightarrow \infty} \sum\limits_{x \in \Lambda_n} \log \int \exp \left\{ \frac{i \lambda}{\sqrt{n}} \left(\eta_x - \tfrac12 \right) H(\tfrac{x}{n}) \right\} d\nu_{1/2}\\
& = \lim_{n \rightarrow \infty} \sum\limits_{x \in \Lambda_n} \log \Big( 1 - \frac{\lambda^2}{8n} H(\tfrac{x}{n})^2 + \mathcal{O}_H (n^{-3/2})  \Big)\\
&= -\frac{\lambda^2}{8}\| H \|_{L^2([0,1])}^2.
\end{align*}
This ends the proof.

\subsection{Convergence of the martingale} In this section, we prove Proposition \ref{prop: martingale convergence}, from which the tightness of the martingale follows directly. By \cite{whitt2007proofs}, we only need to check the following conditions:
\begin{itemize}
	\item[{\rm (i)}] for any $n > 1$, the quadratic variation process $\langle \mathcal{M}^n(H) \rangle_t$ has continous trajectories, a.s.;
	
	\item[{\rm (ii)}] the following limit holds
	\[ \lim\limits_{N \to \infty} \mathbb{E} \Big[ \sup\limits_{0 \leq s \leq T} |\mathcal{M}_s^n(H) - \mathcal{M}_{s-}^n(H)|\Big] = 0; \]
	
	\item[{\rm (iii)}] for any $t \in [0,T]$, the sequence $\{\langle \mathcal{M}^n(H) \rangle_t\}_{n \geq 1}$ converges to $\frac{t C_\alpha}{2} \| \nabla H \|_{L^2{[0,1]}}^2$ in probability.
\end{itemize}

By the calculations in Subsection \ref{subsec: calculation}, the quadratic variation of $\mathcal{M}_t^n(H)$ is given by
\begin{align}
	&\langle \mathcal{M}^n(H) \rangle_t = \int_0^t \sum_{x,y \in \Lambda_n} [n s_\alpha (x-y) + n^{\theta-1} p_\gamma (y-x) \eta_x (1-\eta_y)] (\eta_x - \eta_y)^2 \big[ H(\tfrac{y}{n}) - H(\tfrac{x}{n}) \big]^2 ds\notag\\
	&+ \int_0^t \sum_{x \in \Lambda_n} \Big[  \frac{n}{2} r_\alpha (x) + \frac{n^{\theta-1}}{2} \sum_{y \leq 0} p_\gamma (x-y)(1-\eta_x) + \frac{n^{\theta-1}}{2} \sum_{y \geq n} p_\gamma (y-x)\eta_x \Big]  H(\tfrac{x}{n})^2 ds\label{quad-var}. 
\end{align}
Then condition (i)  immediately holds.

By \eqref{martingale}, 
\[ \sup\limits_{0 \leq t \leq T} |\mathcal{M}_t^n(H) - \mathcal{M}_{t-}^n(H)| = \sup\limits_{0 \leq t \leq T} |\mathcal{Y}_t^n(H) - \mathcal{Y}_{t-}^n(H)| \leq \frac{C(H)}{\sqrt{n}}. \]
So condition (ii)  holds.

It remains to check condition (iii). We first prove the convergence of the expectation for $\langle \mathcal{M}^n(H) \rangle_t$.

\begin{lemma}
For any $H \in \mathcal{S}$,
\[ \lim\limits_{n \to \infty} \mathbb{E} [\langle \mathcal{M}^n(H) \rangle_t]= t \frac{C_\alpha}{2} \| \nabla H \|_{L^2([0,1])}^2  \]
with $C_\alpha = \sum_{y \in \mathbb{Z} \setminus \{0\}}|y|^{1 - \alpha}$.
\end{lemma}

\begin{proof}
By Fubini's theorem, we can pass the expectation inside the integral in \eqref{quad-var}. Consider the first integral in the right-hand side in \eqref{quad-var}. We have
\begin{align*}
& \int_0^t \mathbb{E} \Big[ \sum_{x,y \in \Lambda_n} n s_\alpha (y-x) (\eta_x - \eta_y )^2 \big( H(\tfrac{y}{n}) - H(\tfrac{x}{n}) \big)^2 \Big] ds\\
= & \frac{tn}{2} \sum\limits_{x,y \in \Lambda_n} s_\alpha(y - x) \big(H(\tfrac{y}{n}) - H(\tfrac{x}{n})\big)^2.
\end{align*}
By Taylor's expansion,  the last line can be written as
\[ \frac{t}{2n} \sum\limits_{x,y \in \Lambda_n, \atop y \neq x} \frac{1}{|y - x|^{\alpha - 1}} H'(\tfrac{x}{n})^2 \]
plus an error term which is bounded by
\[ C(H) \frac{t}{n^3} \sum\limits_{x,y \in \Lambda_n, \atop y \neq x} |y-x|^{3 - \alpha}. \]
The latter term vanishes as $n \to \infty$ since $\alpha > 2$, and the former term converges to $\frac{C_\alpha t}{2}  \| \nabla H \|_{L^2([0,1])}^2$.
Furthermore, we can bound 
\begin{align*}
& \int_0^t \mathbb{E} \Big[ \sum\limits_{x,y \in \Lambda_n} n^{\theta - 1} p_\gamma(y - x) \eta_x (1 - \eta_y)  \big(H(\tfrac{y}{n}) - H(\tfrac{x}{n})\big)^2 \Big] ds\\
= &  \frac{tn^{\theta - 1}}{4}  \sum\limits_{x, y \in \Lambda_n} s_\gamma(y - x) \big( H(\tfrac{y}{n}) - H(\tfrac{x}{n}) \big)^2\\
\leq & C t n^{\theta - \gamma} \int\int_{0 \leq u, v \leq 1, |u - v| \geq \frac{1}{n}} \frac{(H(u) - H(v))^2}{|u - v|^{1 + \gamma}} dudv.
\end{align*}
The last line has order $n^{\theta - \gamma}$ if $\gamma < 2$, $n^{\theta-2} \log n$ if $\gamma = 2$, $n^{\theta-2}$ if $\gamma > 2$, which vanishes in the limit $n \rightarrow \infty$ by Assumption \ref{assump: parameters}.

Then we consider the second integral in the right-hand side of \eqref{quad-var}. By \eqref{def-r-alpha}, the first part can be bounded by
\begin{align*}
\frac{nt}{2} \sum\limits_{x \in \Lambda_n} H(\tfrac{x}{n})^2 \Big(\sum\limits_{y \leq 0} s_\alpha(x - y) + \sum\limits_{y \geq n} s_\alpha(x - y) \Big) \leq C nt \sum\limits_{x \in \Lambda_n} H(\tfrac{x}{n})^2  [x^{- \alpha} + (n-x)^{- \alpha}].
\end{align*}
Since $H \in \mathcal{S}$,  for any $ k \in \mathbb{N}$ there is some $C (k, H) > 0$ such that $|H(\frac{x}{n})| \leq C(k, H) |\frac{x}{n}|^k$ for $x \in \Lambda_n$. Thus, choosing $k$ large enough such that $2k-\alpha > 0$, the first term in the last expression is bounded by
\begin{equation}\label{eqn 8}
C(k,H) t n^{1-2k} \sum_{x \in \Lambda_n} x^{2k-\alpha} \leq C (k,H) t n^{2-\alpha},
\end{equation}
which vanishes as $n \to \infty$ since $\alpha > 2$.
The same arguments hold for the other term. By using the same argument as above, the expectation of the remaining part for the second integral in the right-hand side of \eqref{quad-var} can be bounded by
\begin{align*}
& \mathbb{E} \Big[ \int_0^t \sum\limits_{x \in \Lambda_n} \Big(\frac{n^{\theta-1}}{2} \sum_{y \leq 0} p_\gamma (x-y)(1-\eta_x ) + \frac{n^{\theta-1}}{2} \sum_{y \geq n} p_\gamma (y-x)\eta_x \Big) H(\tfrac{x}{n})^2 ds\Big]\\
& \leq C t n^{\theta - 1} \sum_{x \in \Lambda_n}H(\tfrac{x}{n})^2  [x^{- \gamma} + (n-x)^{- \gamma}] \leq C t n^{\theta-\gamma},
\end{align*}
which vanishes as $n \to \infty$ since $\theta < \gamma$.  This ends the proof. 
\end{proof}

Finally, we can verify condition (iii)  by proving the following result.

\begin{lemma}
For any $H \in \mathcal{S}$, 
\[ \lim\limits_{n \to \infty} \mathbb{E} \Big[ \Big( \langle \mathcal{M}^n(H) \rangle_t - \mathbb{E} [\langle \mathcal{M}^n(H) \rangle_t] \Big)^2\Big] = 0, \]
where $\langle \mathcal{M}^n(H) \rangle_t$ is defined in \eqref{quad-var}.
\end{lemma}

\begin{proof}
By Cauchy-Schwarz inequality,  the expectation in the lemma is bounded by some constant times the sum of the following three terms
\begin{align}
& \mathbb{E}  \Big[\Big( \int_0^t \sum\limits_{x,y \in \Lambda_n} ns_\alpha(x - y) M_{x,y}(\eta (s)) \left(H(\tfrac{y}{n}) - H(\tfrac{x}{n})\right)^2 ds\Big)^2 \Big]\label{sMar-1},\\
& \mathbb{E}  \Big[\Big( \int_0^t \sum\limits_{x,y \in \Lambda_n} n^{\theta - 1} p_\gamma(y - x) N_{x,y}(\eta (s)) \left(H(\tfrac{y}{n}) - H(\tfrac{x}{n})\right)^2 ds\Big)^2 \Big]\label{sMar-2},\\
& \mathbb{E} \Big[\Big( \int_0^t \sum\limits_{x \in \Lambda_n} \Big\{ \frac{n^{\theta - 1}}{2}\sum\limits_{y \leq 0} p_\gamma(x - y) - \frac{n^{\theta - 1}}{2}\sum\limits_{y \geq n} p_\gamma(y - x) \Big\} Z_{x}(\eta (s)) H(\tfrac{x}{n})^2 ds\Big)^2 \Big]\label{sMar-3},
\end{align}
where $M_{x,y}(\eta) = (\eta_y - \eta_x)^2 - \frac12$, $N_{x,y}(\eta) = \eta_x(1 - \eta_y) - \frac14$ and $Z_x (\eta) = \frac12 - \eta_x$.

We first prove that \eqref{sMar-1} vanishes as $n \rightarrow \infty$. Notice that, $\nu_{1/2}$ is a product measure on $\Lambda_n$ and $E_{\nu_{1/2}} [M_{x,y}(\eta)]$ $= 0, \, \forall x,y \in \Lambda_n$. By Cauchy-Schwarz inequality, \eqref{sMar-1} is bounded by 
\begin{align*}
&C t^2 n^2 \sum\limits_{x \neq y \in \Lambda_n} \sum\limits_{z \in \Lambda_n\setminus\{x\}}  s_\alpha(x - y) s_\alpha(x - z) \left(H(\tfrac{y}{n}) - H(\tfrac{x}{n})\right)^2 \left(H(\tfrac{z}{n}) - H(\tfrac{x}{n})\right)^2
E_{\nu_{1/2}} [M_{x,y}(\eta) M_{x,z}(\eta)]\\
& \leq C t^2 n^2 \sum\limits_{x \in \Lambda_n}\bigg( \sum\limits_{y \in \Lambda_n\setminus\{x\}} \left(H(\tfrac{y}{n}) - H(\tfrac{x}{n})\right)^2 s_\alpha(x - y) \bigg)^2 \\
&\leq Ct^2n^{3-2\alpha} \int_0^1 \Big( \int_{|v-u| \geq n^{-1}, 0 \leq v \leq 1} \frac{(H(u)-H(v))^2}{|u-v|^{1+\alpha}} dv\Big)^2du.
\end{align*}
Since $[H(u)-H(v)]^2 \leq C(H)(u-v)^2$ and $\alpha > 2$, the last term has order $n^{-1}$ and thus  vanishes as $n \to \infty$.

Similarly, \eqref{sMar-2} can be bounded by 
\begin{align*}
&C t^2 n^{2\theta-2} \sum\limits_{x \in \Lambda_n}\bigg( \sum\limits_{y \in \Lambda_n\setminus\{x\}} \left(H(\tfrac{y}{n}) - H(\tfrac{x}{n})\right)^2 p_\gamma (x - y) \bigg)^2 \\
&\leq Ct^2n^{2 \theta - 1 - 2 \gamma} \int_0^1 \Big( \int_{|v-u| \geq n^{-1}, 0 \leq v \leq 1} \frac{(H(u)-H(v))^2}{|u-v|^{1+\gamma}} dv\Big)^2du.
\end{align*}
Note that the last line has order $n^{2\theta-1-2\gamma}$ if $\gamma < 2$, $n^{2\theta - 5} (\log n)^2$ if $\gamma =2$, $n^{2\theta-5}$ if $\gamma > 2$, which converges to zero by Assumption \ref{assump: parameters}.

For \eqref{sMar-3}, we bound it by
\begin{align*}
C t^2 n^{2\theta-2} \sum\limits_{x \in \Lambda_n}  H(\tfrac{x}{n})^4 [x^{-2\gamma} + (n-x)^{-2\gamma}] \leq C t^2 n^{2\theta-2\gamma-1},
\end{align*}
which vanishes as $n \to \infty$.  In the last inequality, we used the same argument as in \eqref{eqn 8}. This concludes the proof.
\end{proof}

\subsection{Tightness of the integral terms}  In this subsection, we prove the tightness of the sequences $\{ \int_0^t A_s^{n,j}(H)ds, t \in [0,T] \}_{n \in \mathbb{N}},$  $j= 1,...,5$. For this purpose, we first recall Kolmogorov-Centsov's tightness criterion, see \cite[page 64]{karatzas1991brownian} for example.

\begin{proposition}\label{pro-tight-KC}
A sequence of continuous processes $\{ X_t^n, t\in[0,T]\}_{n\in\mathbb{N}}$ is tight with respect to the uniform topology of $\mathcal{C}([0,T]; \mathbb{R})$, if there exist constants $K,a,b>0$ such that
	\begin{equation*}
		E\left[ |X_t^n - X_s^n|^a \right]
		\leq K|t-s|^{1+b}
	\end{equation*}
	for any $s,t\in[0,T]$ and any $n\in\mathbb{N}$.
\end{proposition}

For $j=1$, following the proof of Lemma \ref{lem: A^n_1}, we  have
\[\E \Big[\Big(\int_s^t A_r^{n,1}(H)dr\Big)^2 \Big] \leq C(H) (t-s)^2,\]
so this term is  tight. For $j= 2, 3, 5$, in the proof of Lemma \ref{lem: A^n_2}, we actually showed that 
\begin{equation}\label{eqn 9}
\E \Big[\Big(\int_s^t A_r^{n,j}(H)dr\Big)^2 \Big] \leq o_{n,H} (1) (t-s)^2,
\end{equation}
where $o_{n,H} (1)$ is some constant depending on $H$ and vanishes as $n \rightarrow \infty$, which verifies the tightness of the integral terms for $j=2,3,5$.  

Finally, we prove the tightness of
\[\Big\{ \int_0^t A_s^{n,4}(H)ds, t \in [0,T] \Big\}_{n \in \mathbb{N}}.\]

\subsubsection{The case $\theta < 1$}  For $\theta < 1$, we actually proved in section 4.1 that \eqref{eqn 9} also holds for $j=4$, which proves the tightness.

\subsubsection{The case $1 \leq \theta \leq  \frac32$}  By Lemmas \ref{lem-out-region-K} and \ref{lem-Taylor-high-order}, we only need to show the tightness of the following process
\begin{equation}\label{eqn 3*}
\int_0^t n^{\theta-3/2} \sum\limits_{\substack{|x - y| \leq K \\ x,y\in\Lambda_n}} (x-y)a_\gamma(x-y) H^\prime (\tfrac{x}{n})\bar{\eta}_x(s) \bar{\eta}_y(s) ds,
\end{equation}
where $K \gg n^{\frac{2\theta-2}{2\gamma-1}}$.  On one hand, by Lemmas \ref{lem-one-order-psi}, \ref{lem-Ltheta-differ} and \ref{lem: III},  for any positive integer $\lambda$,
\begin{multline}\label{eqn 10}
	\E \Big[ \Big(\int_s^t n^{\theta-3/2} \sum\limits_{\substack{|x - y| \leq K \\ x,y\in\Lambda_n}} (x-y)a_\gamma(x-y) H^\prime (\tfrac{x}{n})\bar{\eta}_x(r) \bar{\eta}_y(r) dr\Big)^2 \Big] \\
	\leq C (t-s)^2 \Big(\frac{n^{2\theta - 2}}{K} + \frac{K^{\lambda+1}}{n^{4-2\theta+\lambda}}\Big)+ C (t-s) \left( \frac{K^{4 - 2\gamma}}{n^{4 - 2\theta}}
	+ \frac{K}{n^{4 - 2\theta}} + \frac{K^{\lambda+1}}{n^{2-2\theta+\lambda}}\right).
\end{multline}
Let $K = n^{1-\delta}$ and let $\lambda$ be large enough such that
\begin{equation}\label{eqn 11}
2 \theta-2 < 1 - \delta, \quad (\lambda+1)(1-\delta) < 2-2\theta+\lambda,
\end{equation}
then the above bound is of order $o_{n,H} (1) (t-s)^2 + C (t-s)/n^a$ for some $a > 0$.  On the other hand, using Cauchy-Schwarz inequality,
\begin{align*}
& \mathbb{E} \Big[\Big(\int_s^t n^{\theta-3/2} \sum\limits_{\substack{|x - y| \leq K \\ x,y\in\Lambda_n}} (x-y)a_\gamma(x-y) H^\prime (\tfrac{x}{n})\bar{\eta}_x(r) \bar{\eta}_y(r) dr \Big)^2\Big]\\
\leq & (t-s)^2 n^{2\theta - 3} E_{\nu_{1/2}} \Big[\Big( \sum\limits_{\substack{|x - y| \leq K \\ x,y\in\Lambda_n}} (x-y)a_\gamma(x-y) H' (\tfrac{x}{n})\bar{\eta}_x \bar{\eta}_{y}  \Big)^2\Big]\\
\leq & C (t-s)^2 n^{2\theta - 3} \sum\limits_{\substack{|x - y| \leq K \\ x,y\in\Lambda_n}} (x-y)^2 a_\gamma (x-y)^2 H' (\tfrac{x}{n})^2
\leq C (t-s)^2 n^{2\theta - 2}.
\end{align*}
Then, the tightness of the sequence $\{\int_0^t A_s^{n,4}(H)ds, t \in [0,T]\}$ follows from the above two estimates, the following lemma and the trivial fact that $(t-s)^2 \leq C(T) (t-s)^{1+\frac{a}{a+b}}$.

\begin{lemma}
For any $a,b>0$, 
\[ \min\Big\{\frac{t}{n^a}, t^2n^b\Big\} \leq t^{1+\frac{a}{a+b}}.\]
\end{lemma}

\begin{proof}
	If $t \leq n^{-(a+b)}$, then
	\[t^2 n^b = t^{1+\frac{a}{a+b}} (tn^{a+b})^{\frac{b}{a+b}} \leq t^{1+\frac{a}{a+b}}.\]
	If $t > n^{-(a+b)}$, then $n^{-a} < t^{\frac{a}{a+b}}$, and thus
	\[tn^{-a} \leq t^{1+\frac{a}{a+b}}, \]
	concluding the proof.
\end{proof}

\subsubsection{The case $\theta=3/2$} In this case, we take $K=n^{1-\delta}$ and $\lambda > 0$ large enough such that $(2\gamma-1)(1-\delta) > 1$ and that the second condition in \eqref{eqn 11} holds. Note that for $\theta = 3/2$, the first condition in \eqref{eqn 11} cannot hold any more. The problem comes from the bound $n^{\theta-2}/K$ in \eqref{eqn 10}. By Lemma \ref{lem-one-order-psi}, this bound comes from the term
\begin{equation}\label{me}
	 \int_0^t \sum_{x = 1}^{n - K - 1} \sum_{y=1}^K y a_\gamma (y)\psi_x^K(\eta(s)) H'(\tfrac{x}{n}) ds.
\end{equation}
Thus, we only need to prove tightness of the above term.

In Subsection \ref{subsec: critical}, we have shown that
\begin{multline*}
	\E \Big[\Big(\int_s^t 2 \sum_{x = 1}^{n - K - 1} \sum_{y=1}^K y a_\gamma (y)\psi_x^K(\eta(r)) H'(\tfrac{x}{n}) dr - m \int_s^t \sum_{x=1}^{n-1} \psi_x^{\varepsilon n} H^\prime (\tfrac{x}{n}) dr \Big)^2\Big] \\
	\leq C(H) (t-s)^2 \Big(\frac{1}{\varepsilon^2 n^\delta} + \frac{n}{n^{(2\gamma-1)(1-\delta)}}\Big)  + C(H) (t-s) \big(\varepsilon^2 + \varepsilon\big).
\end{multline*}
Taking $K = \varepsilon n$ in Lemma \ref{lem-one-order-psi},  we have
\[	\E \Big[\Big( m \int_s^t \sum_{x=1}^{n-1} \psi_x^{\varepsilon n} H^\prime (\tfrac{x}{n}) dr \Big)^2\Big] \leq C(H) (t-s)^2 \varepsilon^{-1}.\]
Thus, by Cauchy-Schwarz inequality, and since $\varepsilon < 1$,
\begin{multline}\label{eqn 12}
	\E \Big[\Big(\int_s^t 2 \sum_{x = 1}^{n - K - 1} \sum_{y=1}^K y a_\gamma (y)\psi_x^K(\eta(r)) H'(\tfrac{x}{n}) dr  \Big)^2\Big] \\
	\leq C(H) (t-s)^2 \Big(\frac{1}{\varepsilon^2 n^\delta} + \varepsilon^{-1} + o_n (1) \Big)  + C(H) (t-s) \varepsilon.
\end{multline}
Note also that by Lemma \ref{lem-one-order-psi}, the last expectation is bounded by $C(H) (t-s)^2 n^\delta$. Next, we consider the following three cases:
\begin{itemize}
	\item[{\rm (i)}] If $n^{-3\delta} < t-s < 1$, then we take $\varepsilon = (t-s)^{1/3}$,  and bound the term in \eqref{eqn 12} by 
	\[C(H) \big((t-s)^{4/3} + (t-s)^{5/3} \big) \leq C(H) (t-s)^{4/3}.\]
	\item[{\rm (ii)}] If $t -s \leq n^{-3\delta} $, then 
	\[(t-s)^2 n^\delta \leq (t-s)^{5/3} \leq (t-s)^{4/3}.\]
	\item[{\rm (iii)}] If $t-s \geq 1$, then we take $\varepsilon = 1/2$,  and bound the term in \eqref{eqn 12} by  $C(H)(t-s)^2 \leq C(H,T) (t-s)^{4/3}$.
\end{itemize}
In any cases, we have
\[\E \Big[\Big(\int_s^t 2 \sum_{x = 1}^{n - K - 1} \sum_{y=1}^K y a_\gamma (y)\psi_x^K(\eta(r)) H'(\tfrac{x}{n}) dr  \Big)^2\Big] \leq C(H,T)(t-s)^{4/3},\]
thus concluding the proof.

\section{Proof of Proposition \ref{pro: definitions equivalent}}\label{sec: pf definitions equivalent}

In this section, we prove  Proposition \ref{pro: definitions equivalent} and therefore derive the uniqueness for the solution of the martingale problem in Definition \ref{def burgers 2}. Let $\mathcal{Y}_t$ be defined as  the solution of the martingale problem in Definition \ref{def burgers 2}. The target  is to construct an extension $\tilde{\mathcal{Y}}_t \in \mathcal{C} ([0,T]; \mathcal{S}'_{\textrm{Dir}})$ of $\mathcal{Y}_t$, which solve the martingale problem in Definition \ref{def burgers} and satisfies the energy estimations therein. We will follow the strategy  in \cite{bernardin2022equilibrium}. The main difference from the precious work comes from the non-linear term $\mathcal{A}_t$.

By definition,  for any $H \in \mathcal{S}$,
\[\mathcal{M}_t (H) := \mathcal{Y}_t (H) - \mathcal{Y}_0 (H) - A \int_0^t \mathcal{Y}_s (\Delta H) ds - B \mathcal{A}_t (H)\]
is a continuous martingale with quadratic variation 
\[\<\mathcal{M} (H)\>_t = t D \|\nabla H\|_{L^2 ([0,1])}^2,\]
where $\mathcal{A}_t (H)$ is the $L^2 (\P)$-limit of $\mathcal{A}_{0,t}^\varepsilon (H)$ as $\varepsilon \rightarrow 0$. Furthermore, $\{\mathcal{Y}_t \in \mathcal{S}', t \in [0,T]\}$ is stationary in the sense of \eqref{stationary def} and for any $u \in \{0, 1\}$,
\begin{equation}\label{cond-bu}
\lim\limits_{\varepsilon \to 0} \mathbb{E} \Big[ \sup\limits_{0 \leq t \leq T} \Big( \int_0^t \mathcal{Y}_s (\iota_{\varepsilon,u}) ds \Big)^2 \Big] = 0. 
\end{equation}

Consider the space
\[ \mathcal{J}= \left\{\{x_t, t \in [0,T]\}\, \text{is progressively measurable};\; \Big(\mathbb{E}\Big[\int_0^T |x_t|^2 dt\Big]\Big)^{1/2} < \infty \right\}.\]
Then, by \eqref{stationary def} and since $\mathcal{S}$ is dense in $L^2 ([0,1])$, for any $H \in L^2 ([0,1])$, we can define a unique process $\mathcal{Y}_t (H)$ on $\mathcal{J}$, such that $\{\mathcal{Y}_t(H), t\in[0,T]\}$ for $H \in L^2 ([0,1])$ coincides with $\{\mathcal{Y}_t(H), t\in[0,T]\}$ for $H \in \mathcal{S}$, see \cite[Lemma 5.1]{bernardin2022equilibrium} for example. Particularly, $\mathcal{Y}_t(H)$ can be defined for $H \in \tilde{\mathcal{S}}$ and $H \in \mathcal{S}_{\rm Dir}$, where
\[ \tilde{\mathcal{S}} = \{ H \in \mathcal{H}^2([0, 1]): H(0) = H(1) = H'(0) = H'(1) = 0\}, \]
with $\mathcal{H}^2([0,1])$ being the classical Sobolev space $\mathcal{W}^{2,2}([0,1])$ on $[0,1]$. Denote by $\mathcal{B}$  the Banach space of real processes $\{x_t; t \in [0,T]\}$ such that $\mathbb{E}[\sup_{t \in [0,T]} |x_t|^2] < \infty$.

The next result states that the martingale problem in Definition \ref{def burgers 2} can also be defined in $\tilde{\mathcal{S}}$.

\begin{lemma}

Let $\{\mathcal{Y}_t, t \in [0,T]\}$ be defined as in Definition \ref{def burgers 2}. Then, for any $H \in \tilde{\mathcal{S}}$,
\[ \tilde{\mc{M}}_t(H) = \mathcal{Y}_t(H) - \mathcal{Y}_0(H) - A\int_0^t \mathcal{Y}_s(\Delta H) ds - B\mathcal{A}_t(H) \]
is also a continuous martingale with quadratic variation $t \rightarrow tD\|\nabla H\|_{L^2([0,1])}^2$. Furthermore, $\mathcal{Y}_{t}$ satisfies the energy estimates: for any $\delta^\prime < \delta$,
\[ \mathbb{E}\Big[\Big( \mathcal{A}_{s,t}^\delta(H) - \mathcal{A}_{s,t}^{\delta'}(H)\Big)^2 \Big] 
\leq C (t - s) \delta \|\nabla H\|_{L^2([0,1])}^2,\qquad \forall H \in \tilde{\mathcal{S}}, \]
where $\mathcal{A}_{s,t}^\delta$ is defined in \eqref{def-A-delta}.

\end{lemma}

\begin{proof}

For any $H \in \tilde{\mathcal{S}}$, there is a sequence $\{H_\varepsilon\}_{\varepsilon > 0} \subset \mathcal{S}$ such that $\lim\limits_{\varepsilon \to 0} H_\varepsilon^{(k)} = H^k$ in $L^2 ([0,1])$ for $k = 0,1,2$ (c.f. \cite[Lemma A.4]{bernardin2022equilibrium}). Furthermore,
\[ \mc{M}_t(H_\varepsilon) = \mathcal{Y}_t(H_\varepsilon) - \mathcal{Y}_0(H_\varepsilon) - A\int_0^t \mathcal{Y}_s(\Delta H_\varepsilon) ds - B\mathcal{A}_t(H_\varepsilon) \]
is a continuous martingale with quadratic variation $t \rightarrow tD\|\nabla H_\varepsilon\|_{L^2([0,1])}^2$.

With the similar arguments  in \cite[Section 5.1]{bernardin2022equilibrium}, by Doob's inequality and since the process is stationary, the processes $\mathcal{Y}_\cdot (H_\varepsilon)$ and $\mathcal{Y}_\cdot (\Delta H_\varepsilon)$ converge to $\mathcal{Y}_\cdot (H)$, and respectively $\mathcal{Y}_\cdot (\Delta H)$. Moreover, there is a martingale $\tilde{\mc{M}}_\cdot (H)$ such that $\mc{M}_\cdot (H_\varepsilon)$ converges to $\tilde{\mc{M}}_\cdot (H)$ in the space $\mathcal{B}$.

It remains to prove the convergence of $\mathcal{A}_t(H_\varepsilon)$. Notice that $\mathcal{A}_t(H_\varepsilon)$ is the $L^2(\mathbb{P})$-limit of $\mathcal{A}_{0,t}^\delta(H_\varepsilon)$ as $\delta \to 0$ and the energy estimate is satisfied: for any $\delta' < \delta$,
\begin{equation}\label{cond-ene}
\mathbb{E}\Big[\Big( \mathcal{A}_{s,t}^\delta(H_\varepsilon) - \mathcal{A}_{s,t}^{\delta'}(H_\varepsilon) \Big)^2\Big] \leq C (t - s) \delta \|\nabla H_\varepsilon\|_{L^2([0,1])}^2, \qquad \forall\, 0 < s < t. 
\end{equation}
Then, for $\mathcal{A}_{s,t}^\delta(H), H \in \tilde{\mathcal{S}}$, by Cauchy-Schwarz inequality,
\begin{align*}
&\mathbb{E}\Big[\Big( \mathcal{A}_{s,t}^\delta(H) - \mathcal{A}_{s,t}^{\delta'}(H) \Big)^2\Big]\\
\leq & 3 \mathbb{E}\Big[\Big( \mathcal{A}_{s,t}^\delta(H) - \mathcal{A}_{s,t}^{\delta}(H_\varepsilon) \Big)^2\Big]
+ 3 \mathbb{E}\Big[\Big( \mathcal{A}_{s,t}^\delta(H_\varepsilon) - \mathcal{A}_{s,t}^{\delta'}(H_\varepsilon) \Big)^2\Big] + 3 \mathbb{E}\Big[\Big( \mathcal{A}_{s,t}^{\delta'}(H_\varepsilon) - \mathcal{A}_{s,t}^{\delta'}(H) \Big)^2\Big]\\
\leq & C(\delta) \|\nabla(H - H_\varepsilon)\|_{L^2([0,1])}^2 + C (t - s) \delta \|\nabla H_\varepsilon\|_{L^2([0,1])}^2
+ C(\delta') \|\nabla(H_\varepsilon - H)\|_{L^2([0,1])}^2,
\end{align*}
where in the last inequality we used the definition of $\mathcal{A}_{s,t}^\delta$ and \eqref{cond-ene}. Let $\varepsilon \to 0$, then the last line is bound by 
\[ C (t - s) \delta \|\nabla H\|_{L^2([0,1])}^2. \]
Therefore, $\mathcal{A}_t^\delta(H) := \mathcal{A}_{0,t}^\delta(H)$ is a Cauchy sequence in $\mathcal{B}$ and we can define the limiting process
\[ \mathcal{A}_t(H) := \lim\limits_{\delta \to 0} \mathcal{A}_t^\delta(H) \qquad \textrm{in}\; L^2 (\P), \quad \forall\, H \in \tilde{\mathcal{S}}.\]

Now, we prove that \[ \lim_{\varepsilon \to 0} \mathcal{A}_t(H_\varepsilon) = \mathcal{A}_t(H)\] in $L^2(\mathbb{P})$. For any $\delta > 0$, by Cauchy-Schwarz inequality,
\begin{align*}
&\mathbb{E}\Big[\big( \mathcal{A}_t(H_\varepsilon) - \mathcal{A}_t(H) \big)^2\Big] \\
& \leq 3 \mathbb{E}\Big[\big( \mathcal{A}_t(H_\varepsilon) - \mathcal{A}^\delta_t(H_\varepsilon) \big)^2\Big] + 3 \mathbb{E}\Big[\big( \mathcal{A}_t^\delta (H_\varepsilon) - \mathcal{A}_t^\delta (H) \big)^2\Big] + 3 \mathbb{E}\Big[\big( \mathcal{A}_t^\delta (H) - \mathcal{A}_t(H) \big)^2\Big]\\
& \leq C t \delta \|\nabla H_\varepsilon\|_{L^2([0,1])}^2 + C(\delta) \|\nabla (H_\varepsilon - H)\|_{L^2([0,1])}^2 + o_\delta(1),
\end{align*}
where $o_\delta (1) \rightarrow 0$ as $\delta \rightarrow 0$. Letting $\varepsilon \to 0$ first and then $\delta \to 0$, we conclude the proof.
\end{proof}

Define  functions $a, \phi, \tilde{\phi}: \mathbb{R} \rightarrow \mathbb{R}$ as
\begin{align*}
& a(u) = c e^{-\frac{1}{u(1 - u)}}\textbf{1}_{(0,1)}(u), \qquad c = \Big(\int_0^1 e^{-\frac{1}{u(1-u)}} du\Big)^{-1},\\
& \tilde{\phi}(u) = \int_0^u a(t) dt, \qquad \phi(u) = 1 - \int_0^u a(t) dt.
\end{align*}
For any $\beta \in (0,1)$ and $\alpha > (1 - \beta)^{-1}$, let
\[ \phi_{\alpha, \beta}(u) = \phi(\alpha(u - \beta)), \qquad \psi_{\alpha, \beta}(u) = u\phi_{\alpha,\beta}(u). \]
Let $\tilde{\psi}_{\alpha, \beta} (u)= \psi_{\alpha, \beta} (1-u)$, $u \in [0,1]$.

\begin{lemma}

Suppose \eqref{cond-bu} is satisfied. Giving $\beta \in (0,1)$ and $\alpha > (1-\beta)^{-1}$, we have that for $H \in \{\psi_{\alpha, \beta}, \tilde{\psi}_{\alpha, \beta}\}$,
\[ \tilde{\mc{M}}_t(H) = \mathcal{Y}_t(H) - \mathcal{Y}_0(H) - A\int_0^t \mathcal{Y}_s(\Delta H) ds - B \mathcal{A}_t(H)\]
is a continuous martingale with quadratic variation $tD\|\nabla H\|_{L^2}^2$. Moreover, $\mathcal{Y}_t$ satisfies the energy estimate: for any $\delta^\prime < \delta$,
\[ \mathbb{E}\Big[\Big( \mathcal{A}_{s,t}^\delta(H) - \mathcal{A}_{s,t}^{\delta'}(H) \Big)^2\Big] \leq C (t - s) \delta \|\nabla H\|_{L^2([0,1])}^2. \]

\end{lemma}

\begin{proof}

For $\varepsilon > 0$, define 
\[ h_\varepsilon(u) = \begin{cases}
\, \frac{u^2}{2\varepsilon}, \qquad & u \in [0, \varepsilon],\\
\, u - \frac{\varepsilon}{2}, \qquad & u \in [\varepsilon, 1].
\end{cases}
 \]
Then, $\psi_\varepsilon = h_\varepsilon \phi_{\alpha, \beta} \in \tilde{\mathcal{S}}$ and $\psi_\varepsilon^{(k)} \rightarrow \psi_{\alpha, \beta}^{(k)}$ in $L^2$ for $k = 0,1$ as $\varepsilon \to 0$. So 
\[ \tilde{\mc{M}}_t(\psi_\varepsilon) = \mathcal{Y}_t(\psi_\varepsilon) - \mathcal{Y}_0(\psi_\varepsilon) - A\int_0^t \mathcal{Y}_s(\Delta \psi_\varepsilon) ds - B\mathcal{A}_t(\psi_\varepsilon) \]
is a continuous martingale with quadratic variation $t \to tD\|\nabla H\|_{L^2}^2$. Define \[\mathcal{A}_t(\psi_{\alpha, \beta}) = \lim_{\delta \to 0} \mathcal{A}_t^\delta(\psi_{\alpha, \beta}).\] 
Since $\mathcal{A}_{t}^\delta(\psi_\varepsilon)$ satisfies the energy estimate, similar arguments as above show that $\mathcal{A}_t(\psi_\varepsilon) \rightarrow \mathcal{A}_t(\psi_{\alpha, \beta})$ in $\mathcal{B}$. Since $\psi_\varepsilon^{(k)} \rightarrow \psi_{\alpha, \beta}^{(k)}$ for $k = 0,1$,one could check that the martingale $\tilde{\mc{M}}_\cdot (\psi_\varepsilon)$, as $\varepsilon \rightarrow 0$, converges to a martingale denote by $\tilde{\mc{M}}_\cdot (\psi)$, and by  \eqref{cond-bu}, $\mathcal{Y}_\cdot (\psi_\varepsilon)$ converges to $\mathcal{Y}_\cdot (\psi)$ in $\mathcal{B}$. 

The similar results hold for $\tilde{\psi}_{\alpha, \beta}$. Thus, 
\[ \tilde{\mc{M}}_t(H) = \mathcal{Y}_t(H) - \mathcal{Y}_0(H) - A\int_0^t \mathcal{Y}_s(\Delta H) ds - B\mathcal{A}_t(H), \quad H \in \{\psi_{\alpha, \beta}, \tilde{\psi}_{\alpha, \beta}\} \]
is a continuous martingale with quadratic variation $tD\|\nabla H\|_{L^2}^2$, and $\mathcal{Y}_t$ satisfies the energy estimate, which concludes the proof.
\end{proof}

\begin{proof}[Proof of Proposition \ref{pro: definitions equivalent}]
For any $H \in \mathcal{S}_{\rm Dir}$, it can be written as
\[ H = H'(0)\psi_{\alpha, \beta} - H'(1)\tilde{\psi}_{\alpha, \beta} + G_{\alpha, \beta}, \]
where $G_{\alpha, \beta} := H - H'(0)\psi_{\alpha, \beta} + \psi'(1)\tilde{\psi}_{\alpha, \beta}$.  One could check directly that $G_{\alpha, \beta} \in \tilde{S}$. By the last two lemmas, \[\tilde{\mc{M}}^{\alpha, \beta}_t (H) := H'(0) \tilde{\mc{M}}_t (\psi_{\alpha, \beta}) - H'(1) \tilde{\mc{M}}_t (\tilde{\psi}_{\alpha, \beta}) + \tilde{\mc{M}}_t(G_{\alpha, \beta})\] is a continuous martingale,
 and \[\mathcal{A}_t(H) := H'(0) \mathcal{A}_t(\psi_{\alpha, \beta}) - H'(1) \mathcal{A}_t(\tilde{\psi}_{\alpha, \beta}) + \mathcal{A}_t(G_{\alpha, \beta})\]
is well defined. This implies that
\[ \tilde{\mc{M}}_t^{\alpha, \beta}(H) = \mathcal{Y}_t(H) - \mathcal{Y}_0(H) - A\int_0^t \mathcal{Y}_s(\Delta H) ds - B\mathcal{A}_t(H). \]

Let $\beta = 1 / \alpha$, then one could check directly that $\|\nabla \psi_{\alpha, \alpha^{-1}}\|_{L^2 ([0,1])}$ and  $\|\nabla \tilde{\psi}_{\alpha, \alpha^{-1}}\|_{L^2 ([0,1])}$ vanish as $\alpha \to \infty$.  In particular, \[\lim_{\alpha \to \infty}\|\nabla (G_{\alpha, \beta} - H)\|_{L^2 ([0,1])} = 0.\]
Using Doob's inequality, it is easy to see that $\{\tilde{\mc{M}}^{\alpha, \alpha^{-1}}_t, t \in [0,T]\}_{\alpha > 0}$ is a Cauchy sequence in $\mathcal{B}$, and thus it converges to a martingale denoted by $\tilde{\mc{M}}_\cdot (H)$ whose quadratic variation is given by 
\[\langle \tilde{\mc{M}}(H) \rangle_t = \lim_{\alpha \to \infty} \langle \tilde{\mc{M}}^{\alpha, \alpha^{-1}}(H) \rangle_t = t D \|\nabla H\|_{L^2}^2.\]
This concludes the proof.
\end{proof}

\bibliographystyle{plain}
\bibliography{zhaoreference.bib}

\begin{thebibliography}{10}

\bibitem{bernardin2022equilibrium}
C.~Bernardin, P.~Goncalves, M.~Jara, and S.~Scotta.
\newblock Equilibrium fluctuations for diffusive symmetric exclusion with long
  jumps and infinitely extended reservoirs.
\newblock In {\em Annales de l'Institut Henri Poincare (B) Probabilites et
  statistiques}, volume~58, pages 303--342. Institut Henri Poincar{\'e}, 2022.

\bibitem{bernardin2019slow}
C.~Bernardin, P.~Goncalves, and B.~Jim{\'e}nez-Oviedo.
\newblock Slow to fast infinitely extended reservoirs for the symmetric
  exclusion process with long jumps.
\newblock {\em Markov Processes And Related Fields}, 25(2):217--274, 2019.

\bibitem{bernardin2021microscopic}
C.~Bernardin, P.~Gon{\c{c}}alves, and B.~Jim{\'e}nez-Oviedo.
\newblock A microscopic model for a one parameter class of fractional
  laplacians with {Dirichlet} boundary conditions.
\newblock {\em Archive for Rational Mechanics and Analysis}, 239(1):1--48,
  2021.

\bibitem{bernardin2017fractional}
C.~Bernardin and B.~Jim{\'e}nez-Oviedo.
\newblock Fractional {Fick’s} law for the boundary driven exclusion process
  with long jumps.
\newblock {\em ALEA}, 14(1):473--501, 2017.

\bibitem{bertini1997stochastic}
L.~Bertini and G.~Giacomin.
\newblock Stochastic {Burgers} and {KPZ} equations from particle systems.
\newblock {\em Communications in mathematical physics}, 183(3):571--607, 1997.

\bibitem{corwin2012kpzun}
I.~Corwin.
\newblock The {Kardar-Parisi-Zhang} equation and universality class.
\newblock {\em Random Matrices: Theory and Applications}, page 76pp, 2012.

\bibitem{corwin2022some}
I.~Corwin.
\newblock Some recent progress on the stationary measure for the open {KPZ}
  equation.
\newblock {\em Toeplitz Operators and Random Matrices: In Memory of Harold
  Widom}, pages 321--360, 2022.

\bibitem{corwin2018open}
I.~Corwin and H.~Shen.
\newblock Open {ASEP} in the weakly asymmetric regime.
\newblock {\em Communications on Pure and Applied Mathematics},
  71(10):2065--2128, 2018.

\bibitem{diehl2017kardar}
J.~Diehl, M.~Gubinelli, and N.~Perkowski.
\newblock The {Kardar--Parisi--Zhang} equation as scaling limit of weakly
  asymmetric interacting {Brownian} motions.
\newblock {\em Communications in Mathematical Physics}, 354(2):549--589, 2017.

\bibitem{gonccalves2014nonlinear}
P.~Gon{\c{c}}alves and M.~Jara.
\newblock Nonlinear fluctuations of weakly asymmetric interacting particle
  systems.
\newblock {\em Archive for Rational Mechanics and Analysis}, 212(2):597--644,
  2014.

\bibitem{gonccalves2018density}
P.~Gon{\c{c}}alves and M.~Jara.
\newblock Density fluctuations for exclusion processes with long jumps.
\newblock {\em Probability Theory and Related Fields}, 170(1-2):311--362, 2018.

\bibitem{gonccalves2020derivation}
P.~Gon{\c{c}}alves, N.~Perkowski, and M.~Simon.
\newblock Derivation of the stochastic {Burgers} equation with {Dirichlet}
  boundary conditions from the {WASEP}.
\newblock {\em Annales Henri Lebesgue}, 3:87--167, 2020.

\bibitem{hairer2013solving}
M.~Hairer.
\newblock Solving the {KPZ} equation.
\newblock {\em Annals of mathematics}, pages 559--664, 2013.

\bibitem{karatzas1991brownian}
I.~Karatzas and S.~Shreve.
\newblock {\em Brownian motion and stochastic calculus}, volume 113.
\newblock Springer Science \& Business Media, 1991.

\bibitem{kardar1986dynamic}
M.~Kardar, G.~Parisi, and Y.~Zhang.
\newblock Dynamic scaling of growing interfaces.
\newblock {\em Physical Review Letters}, 56(9):889, 1986.

\bibitem{klscaling}
C.~Kipnis and C.~Landim.
\newblock {\em Scaling limits of interacting particle systems}, volume 320.
\newblock Springer Science \& Business Media, 2013.

\bibitem{Mateski2021KPZ}
K.~Matetski, J.~Quastel, and D.~Remenik.
\newblock The {KPZ} fixed point.
\newblock {\em Acta Mathematica}, 227(1):115--203, 2021.

\bibitem{mitoma1983tightness}
I.~Mitoma.
\newblock Tightness of probabilities on ${C} ([0,1];\mathscr{Y}')$ and ${D}
  ([0,1];\mathscr{Y}')$.
\newblock {\em The Annals of Probability}, 11(4):989--999, 1983.

\bibitem{quastel2011kpz}
J.~Quastel.
\newblock Introduction to {KPZ}.
\newblock {\em Current Developments in Mathematics, 2011}, pages 125--194,
  2012.

\bibitem{Quastel2023longKPZ}
J.~Quastel and S.~Sarkar.
\newblock Convergence of exclusion processes and the {KPZ} equation to the
  {KPZ} fixed point.
\newblock {\em Journal of the American Mathematical Society}, 36(1):251--289,
  2023.

\bibitem{sethuraman2016microscopic}
S.~Sethuraman.
\newblock On microscopic derivation of a fractional stochastic {Burgers}
  equation.
\newblock {\em Communications in Mathematical Physics}, 341(2):625--665, 2016.

\bibitem{whitt2007proofs}
W.~Whitt.
\newblock Proofs of the martingale {FCLT}.
\newblock {\em Probability Surveys}, 4:268--302, 2007.

\end{thebibliography}

\end{document}